\documentclass{amsart}
\usepackage{hyperref,enumerate}
\usepackage{xcolor}

\newtheorem{theorem}{Theorem}[section]
\newtheorem{lemma}[theorem]{Lemma}
\newtheorem{proposition}[theorem]{Proposition}
\newtheorem{corollary}[theorem]{Corollary}

\theoremstyle{definition}
\newtheorem{definition}[theorem]{Definition}

\theoremstyle{remark}
\newtheorem{remark}[theorem]{Remark}
\numberwithin{equation}{section}

\title[Extremal Metrics for $P_g$]{Extremal metrics for the Paneitz Operator on closed four-Manifolds}

\author{Samuel P\'{e}rez-Ayala}
\address{Department of Mathematics\\
         Princeton University\\
         Princeton, NJ 08544}

\email{\href{sperezayala@princeton.edu}{sperezayala@princeton.edu}}

\begin{document}

\begin{abstract}
Let $(M^4,g)$ be a closed Riemannian manifold of dimension four. We investigate the properties of metrics which are critical points of the eigenvalues of the Paneitz operator when considered as functionals on the space of Riemannian metrics with fixed volume. We prove that critical metrics of the aforementioned functional restricted to conformal classes are associated with a higher-order analog of harmonic maps (known as extrinsic conformal-harmonic maps) into round spheres. This extends to four-manifolds well-known results on closed surfaces relating metrics maximizing laplacian eigenvalues in conformal classes with the existence of harmonic maps into spheres. The case of general critical points (not restricted to conformal classes) is also studied, and partial characterization of these is provided. 
\end{abstract}


\maketitle

\section{Introduction}\label{PaneitzIntro}

Let $M^4$ be a closed (compact, no boundary) four-dimensional Riemannian manifold. In 1983, S.M. Paneitz (\cite{Paneitz}) introduced a fourth-order elliptic operator, nowadays known as the Paneitz operator, acting on arbitrary pseudo-Riemannian manifolds. For a Riemannian metric $g$, it is denoted by $P_g$ and defined on smooth functions by
\begin{equation}\label{P-definition}
    P_g(\phi):= \Delta_g^2(\phi) - \text{div}_g\left(\frac{2}{3}R_g g - 2\text{Ric}_g\right)\nabla_g\phi.
\end{equation}
It was the first higher-order example of a conformally covariant operator. Indeed, for a conformal metric $g_w=e^{2w}g\in [g]$, $P_g$ transforms as 
\begin{equation}\label{ConfTransP}
    P_{g_w}(\phi) = e^{-4w}P_g(\phi).
\end{equation}
By choosing a local orthonormal basis $\{E_1,\cdots,E_4\}$, $P_g$ can be rewritten locally as
\begin{equation}\label{P-local}
    P_g(\phi) = \Delta_g^2(\phi) - \frac{2}{3}\text{div}_g(R_g\nabla_g\phi) + 2\sum_{i=1}^4\text{div}_g(\text{Ric}_g(\nabla_g\phi,E_i)E_i),
\end{equation}
and thus
\[
    \langle P_g(\phi), \psi\rangle_{L^2(M^4,g)} = \int_M\left\{(\Delta_g\phi)(\Delta_g\psi) + \frac{2}{3}R_g\langle\nabla_g\phi,\nabla_g\psi\rangle - 2\text{Ric}_g(\nabla_g\phi,\nabla_g\psi)\right\}\;dv_g.
\]
Finally, since $M^4$ is assumed to be compact and $P_g$ is an elliptic operator, its spectrum forms a discrete sequence of real numbers. We arrange them as follows:
\begin{equation}\label{eigenvaluesP}
    \lambda_1(P_g)\le \lambda_2(P_g) \le \lambda_3(P_g)\le \cdots \le \lambda_k(P_g) \nearrow \infty,
\end{equation}
where each eigenvalue is repeated according to their multiplicities. 

\subsection{Extremal metrics on conformal classes - main results}
There are many conformal invariants associated with $P_g$ that are relevant to our work. As in the case of surfaces with the Laplace-Beltrami operator, the associated energy for the Paneitz operator is conformally invariant, that is, 
\begin{equation}\label{energy}
    \langle P_{g_w}(\phi),\psi\rangle_{L^2(M^4,g_w)}=\langle P_g(\phi), \psi\rangle_{L^2(M^4,g)}
\end{equation}
for all $g_w\in [g]$ as can be checked by using (\ref{ConfTransP}). Another consequence of (\ref{ConfTransP}) is the conformal invariance of the dimension of $\text{ker}(P_g)$. In fact, the following set identity holds: 
\begin{equation}\label{kernelP}
\text{Ker}(P_{g_w}) = \text{Ker}(P_g).
\end{equation}
Finally, the number of negative eigenvalues $N([g])$ of $P_g$ is also an invariant of $[g]$ as shown in \cite{Canzani2}. 

From (\ref{ConfTransP}), we deduce that the eigenvalues of the Paneitz operator scale like $c^{-2}$ when the metric $g$ is multiplied by a positive constant $c$: $\lambda_k(P_{cg}) = c^{-2}\lambda_k(P_g)$. Therefore, in order to study possible extremal metrics for the $k$-th eigenvalue functional associated with $P_g$,
\begin{equation}\label{P-functional2}
    g_w:=e^{2w}g\in[g] \mapsto \lambda_k(P_{g_w}),
\end{equation}
some sort of normalization is required. We will argue that constraining the volume is enough and more natural in this context. To this end, fixed a background metric $g\in [g]$, and set
\begin{equation}\label{C-subset}
C:= C(g) = \left\{g_w\in [g]: \text{Vol}(M^4,g_w)=\text{Vol}(M^4,g)\right\}.
\end{equation}
Extremal metrics for the constrained $k$-th eigenvalue functional
\begin{equation}\label{C-PFunctional}
g_w=e^{2w}g\in C\mapsto \lambda_k(P_{g_w})
\end{equation}
are in correspondence with extremal metrics of the scale-invariant functional
\begin{equation}\label{C-PFunctional3}
g_w\in[g] \mapsto \lambda_k(P_{g_w})\text{Vol}(M^4,g_w).
\end{equation}
We will often refer to either (\ref{C-PFunctional}) or (\ref{C-PFunctional3}) as the \textit{normalized $k$-th eigenvalue functional associated with $P_g$}.

We shall call a conformal metric \textit{conformally extremal for the k-th normalized eigenvalue functional} or \textit{conformally extremal for $\lambda_k$} if it is a critical point of the scale-invariant functional (\ref{C-PFunctional3}) (see Definition \ref{C-extremal}). We remark that, in the case of the Laplace-Beltrami operator on closed surfaces, the terminology \textit{C-extremal} is used by El Soufi-Ilias in \cite{Soufi} to refer to critical points of the constrained eigenvalue functional, while the term \textit{conformally maximal} is used in \cite{Karpukhin2} for global maximizers in conformal classes.

Our results concern the necessary and sufficient conditions for the existence of these special metrics. We start with a discussion of some necessary conditions. 

\begin{theorem}\label{Main1} Let $(M^4,g)$ be a four-dimensional closed Riemannian manifold, and suppose that $g_e\in[g]$ is conformally extremal for $\lambda_k$. If $\lambda_k(P_{g_e})\not = 0$, then there exists a finite family of eigenfunctions $\{\phi_1,\cdots,\phi_p\}$ associated with $\lambda_k(P_{g_e})$ satisfying $\sum_{i=1}^p \phi_i^2 \equiv 1$. Moreover,
\begin{equation}\label{Main1Formula}
\lambda_k(P_{g_e}) = \sum_{i=1}^p \left\{\phi_i\Delta_{g_e}^2\phi_i + \frac{2}{3}R_{g_e}|\nabla_{g_e}\phi_i|^2 - 2\text{Ric}_{g_e}(\nabla_{g_e}\phi_i, \nabla_{g_e}\phi_i)\right\},
\end{equation}
and $\lambda_k(P_{g_e})$ is degenerate.
\end{theorem}

Let us make a few remarks about what the extremal condition $\sum_{i=1}^p\phi^2_i = 1$ means. Suppose we consider the functional $g_w\in [g] \mapsto$ $\lambda_k(-\Delta_{g_w})\text{Area}(M^2,g_w)$ on any closed surface for $k\ge 2$, where $\lambda_k(-\Delta_{g_w})$ denotes the $k$-th eigenvalue of the Laplace-Beltrami operator\footnote{In most works dealing with laplacian eigenvalues, it is common to denote the first laplacian eigenvalue by $\lambda_0(-\Delta_g)$ as it is always zero. We do not adopt this convention here: to stay consistent with our notation, here $\lambda_k(-\Delta_g)$ denotes the actual $k$-th eigenvalue of $-\Delta_g$.}. In this case, if $g_{e}\in [g]$ is conformally extremal, then there exists a collection $\{\phi_1,\cdots,\phi_p\}$ of eigenfunctions associated with $\lambda_k(-\Delta_{g_e})$ such that $\sum_{i=1}^p\phi_i^2=1$ (see Theorem 4.1 (i) in \cite{Soufi2}). This means that $\Phi = (\phi_1,\cdots, \phi_p): (M^2,g_e) \rightarrow \mathbb{S}^{p-1}$ is a harmonic map with constant energy density satisfying $2e_{g_e}(\Phi):= |\nabla_{g_e} \Phi|^2=\lambda_k(-\Delta_{g_{e}}) $. Moreover, the multiplicity of $\lambda_k(-\Delta_{g_{e}})$ is at least $2$, that is, $\lambda_k(-\Delta_{g_{e}})$ is degenerate. These types of results were first observed by Nadirashvili in \cite{Nadirashvili3}.

Theorem \ref{Main1} is an extension of this result to the case of the Paneitz operator on closed four-manifolds. In what follows, we explain the meaning of the map $\Phi = (\phi_1,\cdots, \phi_p):(M^4,g_e) \rightarrow \mathbb{S}^{p-1}$ in our setting. Let $N$ be a closed Riemannian manifold that is isometrically embedded into $\mathbb{R}^n$ for some $n$ sufficiently large, and consider the Sobolev space
\begin{equation}
    W^{2,2}(M^4,N):= \{U\in W^{2,2}(M^4,\mathbb{R}^n): U(x)\in N \text{ a.e. }x\in M\}
\end{equation}
In Angelsberg's work \cite{Angelsberg}, the author introduced a fourth-order energy functional 
\[\mathbb{P}_4:W^{2,2}(M^4,N)\rightarrow \mathbb{R}
\] 
defined by 
\begin{equation}
\begin{split}
    \mathbb{P}_4[U]& := \int_M\left\{|\Delta_g U|^2 + \frac{2}{3}R_g|\nabla_gU|^2 -2\text{Ric}_g(\nabla_gU,\nabla_gU) \right\}\;dv_g \\ &= \sum_{i=1}^n\int_M\left\{(\Delta_g U_i)^2 + \frac{2}{3}R_g|\nabla_gU_i|^2 -2\text{Ric}_g(\nabla_gU_i,\nabla_gU_i) \right\}\;dv_g.
\end{split}
\end{equation}
Critical points with respect to compactly supported variations in the target manifold $N$ corresponds to solutions of the system of partial differential equations
\begin{equation}\label{PaneitzEq}
    P_g(U)\perp T_UN
\end{equation}
in the sense of distributions; see Section 1.4 in \cite{Angelsberg}. This type of system of nonlinear PDEs (in the intrinsic case) have been studied before and its solutions are known as \textit{conformal-harmonic maps}; see the work by B\'erard, Biquard-Madani, and Lin-Zhu in \cite{Berard, Biquard, Lin}, respectively, and references therein for more details. We will adopt this terminology for solutions of (\ref{PaneitzEq}).

If $N = \mathbb{S}^{p-1}$, so that $U$ is a sphere-valued map, equation (\ref{PaneitzEq}) becomes
\begin{equation}\label{PaneitzEq2}
P_g(U) = e_g(U)U,
\end{equation}
where $e_g$ denotes the density function
\begin{equation}
e_g(U) = \sum_{i=1}^p\left\{U_i\Delta_g^2 U_i + \frac{2}{3}R_g |\nabla_g U_i|^2 - 2\text{Ric}_g(\nabla_g U_i, \nabla_g U_i)\right\}.
\end{equation}
Now, consider the collection of eigenfunctions $\{\phi_1,\cdots,\phi_p\}$ associated with $\lambda_k(P_{g_e})$ provided by Theorem \ref{Main1}, and define the map $\Phi:M^4\rightarrow \mathbb{S}^{p-1}$ by setting $\Phi:=(\phi_1,\cdots,\phi_p)$. Note that the condition $\sum_{i=1}^p\phi^2=1$ implies that $\Phi$ is a well-defined map into $\mathbb{S}^{p-1}$. From (\ref{Main1Formula}) in Theorem \ref{Main1}, we then have
\begin{equation}
    P_{g_e}(\Phi) = \lambda_k(P_{g_e}) \Phi \;\;\;\text{and}\;\;\; e_{g_e}(\Phi) = \lambda_k(P_{g_e}),
\end{equation}
which means that $\Phi\in W^{2,2}(M^4,\mathbb{S}^{p-1})$ solves (\ref{PaneitzEq2}) and it has constant density $e_{g_e}(\Phi)$ equal to $\lambda_k(P_{g_e})$. Therefore, if $g_e$ is a conformally extremal metric as in Theorem \ref{Main1}, then $(M^4,g_e)$ admits a conformal-harmonic map into a sphere. The next corollary summarizes our discussion. 

\begin{corollary}\label{Cor-Main1}
Suppose we are under the same assumptions as in Theorem \ref{Main1}. Then the map $\Phi=(\phi_1,\cdots,\phi_p): (M^4,g_e)\rightarrow \mathbb{S}^{p-1}$ is a conformal-harmonic map. 
\end{corollary}

The next theorem provides a partial converse to Theorem \ref{Main1} in two ways. Similar results have been proved for the case of the Laplace-Beltrami operator; see Theorem 2.2 in \cite{Soufi} and Theorem 2.3 in \cite{Karpukhin4}. 

\begin{theorem}\label{Main2}
Let $(M^4,g)$ be a four-dimensional closed Riemannian manifold. 
\begin{enumerate}[(i)]
\item Assume there is a metric $g_e\in [g]$ such that $\lambda_k(P_{g_e})\not = 0$ and for which either $\lambda_k(P_{g_e})>\lambda_{k-1}(P_{g_e})$ or $\lambda_k(P_{g_e})<\lambda_{k+1}(P_{g_e})$ holds. If there exists a family of eigenfunctions $\{\phi_1,\cdots,\phi_p\}$ associated with $\lambda_k(P_{g_e})$ such that $\sum_{i=1}^p \phi_i^2$ is constant, then $g_e$ is conformally extremal for $\lambda_k$.

\item Assume there is a smooth conformal-harmonic map $\Phi:(M^4, g) \rightarrow (\mathbb{S}^{p-1},g_r)$ with energy density satisfying either $e_g(\Phi)>0$ or $e_g(\Phi)<0$ pointwise on $M^4$. Then there exists a smooth metric $g_\Phi$ in $[g]$ such that $g_\Phi$ is conformally extremal for $\lambda_k$ for some $k$, and $e_{g_\Phi}(\Phi) = \lambda_k(P_{g_\Phi})$. 
\end{enumerate}
\end{theorem}

In the case of the Laplace operator on closed surfaces, the existence of maximal metrics in conformal classes has been shown; see for instance the work of Petrides and 
Karpukhin-Nadirashvili-Penskoi-Polterovich in \cite{Petrides1,Petrides2, Karpukhin2}. In their work, the connection with harmonic maps is used, and the regularity of maximal metrics is understood in light of the regularity of harmonic maps. Because of Theorem \ref{Main2}, part (ii), we expect the connection with conformal-harmonic maps to be relevant in understanding the general existence and regularity theory for maximal metrics for the Paneitz operator. A relevant question to us, for instance, is how large can be the zero set of the energy density. 

In the work of Ammann-Jammes \cite{Ammann}, it is shown that for conformally covariant operators (e.g. GJMS operators) whose order is strictly less than the dimension of the manifold, the supremum of positive normalized eigenvalues in conformal classes is always infinity (see Theorem 1.3 in \cite{Ammann} for precise statement). However, Ammann-Jammes' result does not apply to our case as the Paneitz operator is of fourth order and we are on a four-dimensional manifold. In fact, we have the following result:

\begin{theorem}\label{Main3}
Let $M^4$ be a four-dimensional closed Riemannian manifold equipped with a conformal class $[g]$. Then there exists a constant $A=A([g])>0$, depending only on the conformal class of $g$, such that
\begin{equation}\label{5.3conclusion}
\sup_{g_w\in [g]}\lambda_k(P_{g_w})\text{Vol}(M^4,g_w)\le A(k+1) <\infty
\end{equation}
for each $k\in\mathbb{N}$. 
\end{theorem}

As for surfaces, Theorem \ref{Main3} argues that it makes sense to seek conformally extremal metrics by maximizing the normalized $k$-th eigenvalue functional over $[g]$. Theorem \ref{Main3} is analogous to results due to Korevaar in \cite{Korevaar} for the Laplace operator in any dimension (see also \cite{Hassannezhad}). It is important to remark that on closed surfaces there are topological upper bounds for the normalized eigenvalue functional associated with $-\Delta_g$, as shown by Yang-Yau \cite{Yang} for the first non-trivial eigenvalue, and by Korevaar \cite{Korevaar} for all eigenvalues on orientable surfaces; see also \cite{Karpukhin3} for the case of non-orientable surfaces. Topological upper bounds for laplacian eigenvalues are no longer possible on closed manifolds of dimension at least three (\cite{Colbois2}). We will return to this point in section \ref{E-MainResults}.

We finish our discussion on conformally extremal metrics with an overview of some specific conformal classes. Similar to the case of closed surfaces, e.g. Corollary 2.1 in \cite{Soufi}, part (i) of Theorem \ref{Main2} allows us to provide examples of conformally extremal metrics. 

\begin{corollary}\label{Homogeneous}
Let $(M^4,g)$ be a homogeneous four-dimensional closed Riemannian manifold. The metric $g$ is conformally extremal for any eigenvalue for which $\lambda_k(P_g)\not = 0$, and for which either $\lambda_k(P_g)>\lambda_{k-1}(P_g)$ or $\lambda_k(P_g) < \lambda_{k+1}(P_g)$ hold.
\end{corollary}

There are only a few examples available of conformally maximal metrics. Due to results by Xu-Yang in \cite{Xu}, the standard metric on $\mathbb{S}_{\frac{1}{\sqrt{2}}}^2\times \mathbb{S}_{\frac{1}{\sqrt{2}}}^2$ is a global maximizer in its conformal class; see Theorem 3.4 there. Another example of a conformally maximal metric is the case of the sphere $\mathbb{S}^4$ equipped with the standard round metric $g_r$. By a result of Gursky in \cite{Gursky} (see also \cite{Gursky2}), we know that $P_{g_r}\ge 0$ and $\text{Ker}(P_{g_r}) = \{\text{constants}\}$. In fact, here $P_{g_r}$ takes the much simpler form 
\begin{equation}\label{PSphere}
    P_{g_r} = (-\Delta_{g_r})\circ (-\Delta_{g_r} + 2) = \Delta_{g_r}^2 - 2\Delta_{g_r}.
\end{equation}
As discussed in \cite{Chang}, a similar formula holds for Einstein metrics.

In the case of the Laplace-Beltrami operator on the $2$-sphere, the metric with the biggest possible first non-zero eigenvalue among all metrics with fixed area equal to $\omega_2=4\pi$ is the round metric. This is due to Hersch (\cite{Hersch}). Our next result generalizes Hersch's result to the four-dimensional round sphere. This result has also been obtained by Gonz\'alez-S\'aez in \cite{Gonzalez} via similar techniques.
\begin{theorem}\label{Main4}
Consider the round sphere $(\mathbb{S}^4,g_r)$, and let $\omega_4=\text{Vol}(\mathbb{S}^4,g_r)$. For any $g_w=e^{2w}g_r\in [g_r]$ satisfying $\text{Vol}(\mathbb{S}^4,g_w) = \omega_4$, we have
\begin{equation}
\lambda_2(P_{g_w})\le 24,
\end{equation}
with equality if and only if $g_w$ is isometric to $g_r$. In particular,
\begin{equation}
\sup_{g_w\in [g_r]} \lambda_2(P_{g_w})\text{Vol}(\mathbb{S}^4,g_w) =\lambda_2(P_{g_r})\omega_4 = 24\omega_4.
\end{equation} 
\end{theorem}

In the case of the Laplace-Beltrami operator on $(\mathbb{S}^2,g_r)$, Karpukhin-Nadirashvili-Penskoi-Polterovich proved in \cite{Karpukhin} that $g_r$ does \textit{not} maximize any higher-order eigenvalue. In fact, the normalized $k$-th eigenvalue functional on $\mathbb{S}^2$, $k\ge 3$, is attained only in the limit by the union of $k-1$ identical touching spheres; see Theorem 1.2 in \cite{Karpukhin}. Whether or not the same occurs on $(\mathbb{S}^4,g_r)$ with the Paneitz operator is still unknown. However, following the ideas of \cite{Petrides3, Nadirashvili4, Kim}, we prove the following Hersch's type inequality for the third eigenvalue on $\mathbb{S}^4$:
\begin{theorem}\label{third-eigen}
For any metric $g_w\in [g_r]$ with $\text{Vol}(\mathbb{S}^4,g_w)=\omega_4$, we have 
\begin{equation}\label{third-eigen1}
\lambda_3(P_{g_w})< 2\cdot 24.
\end{equation}
In particular, 
\begin{equation}\label{third-eigen2}
\sup_{g_w\in [g_r]}\lambda_3(P_{g_w})\text{Vol}(\mathbb{S}^4,g_w) \le 2\cdot 24\omega_4.
\end{equation}
\end{theorem}
Similar to the case of $\mathbb{S}^2$ for laplacian eigenvalues, we expect the equality in \ref{third-eigen2} to be attained by a sequence of conformal metrics degenerating to the disconnected union of two standard spheres. This would imply the non-existence of smooth maximizers for the normalized third eigenvalue functional on $(\mathbb{S}^4,g_r)$.

\subsection{Extremal metrics on $\mathcal{R}(M^4)$ - main results}\label{E-MainResults} We use $\mathcal{R}(M^n)$ to denote the space of all Riemannian metrics on the closed manifold $M^n$ of dimension $n$. We shall call a metric $g$ \textit{extremal for the normalized k-th eigenvalue functional} or \textit{extremal for $\lambda_k$} if it is a critical point for $\bar g\in\mathcal{R}(M^4)\mapsto \lambda_k(P_{\bar g})\text{Vol}(M^4,\bar g)$ (see Definition \ref{Extremal}).  As mentioned, on closed surfaces there are topological upper bounds for the normalized eigenvalue functional. We do not know yet if the analogous result holds for the Paneitz operator, and it would be interesting to investigate whether that is the case. However, properties of extremal metrics can still be studied. 

In the case of closed surfaces, critical points of the normalized eigenvalue functional $g\in \mathcal{R}(M^2)\mapsto \lambda_k(-\Delta_g)\text{Area}(M^2,g)$ over the space of all Riemannian metrics are associated with the existence of isometric minimal immersions into spheres. Indeed, if $g$ is extremal for the $k$-th normalized eigenvalue functional, then there exists a collection $\{\phi_1,\cdots,\phi_p\}$ of eigenfunctions associated with $\lambda_k(-\Delta_g)$ such that $\sum_{i=1}^p\phi_i^2 = \frac{2}{\lambda_k(-\Delta_g)}$ is constant and $d\phi_i\otimes d\phi_i = g$. This condition is almost sufficient for the existence of extremal metrics (see Theorem 3.1 in \cite{Soufi2} for precise statement). As realized by Takahashi in \cite{Takahashi}, it turns out that $\sum_{i=1}^pd\phi_i\otimes d\phi_i = g$ means that the map $\Phi = (\phi_1,\cdots,\phi_p): (M^2,g)\rightarrow \mathbb{S}^{p-1}\left(\sqrt{\frac{2}{\lambda_k(-\Delta_g)}}\right)$ is an isometric immersion whose image is minimal. 

To understand our results, let us point out that the condition $\sum_{i=1}^pd\phi_i\otimes d\phi_i = g$ can be understood in terms of the vanishing of the covariant two-tensor $s_g(\phi):= d\phi\otimes d\phi - \frac{1}{2}|\nabla_g\phi|^2g$, known as the \textit{stress-energy tensor}. In this language, if $g$ is extremal for the $k$-th normalized eigenvalue functional, then one can find a collection $\{\phi_1,\cdots,\phi_p\}$ of eigenfunctions associated with $\lambda_k(-\Delta_g)$ such that $\sum_{i=1}^ps_g(\phi_i)=0$; see proof of Proposition 2.1 in \cite{Fraser}. 

Our main result concerning extremal metrics for the Paneitz operator over $\mathcal{R}(M^4)$ is an analog of this result. Theorem \ref{E-Main} below argues that extremal metrics for the normalized $k$-th eigenvalue functional $g\in\mathcal{R}(M^4)\mapsto \lambda_k(P_g) \text{Vol}(M^4,g)$ can be (partially) characterized in terms of a similar condition on a trace-free covariant two-tensor $\tau_g(\phi)$ defined in section \ref{ExtremalMetrics}; see (\ref{Two-Tensor}).

\begin{theorem}\label{E-Main}
Let $g\in \mathcal{R}(M^4)$ be a Riemannian metric for which $\lambda_k(P_g)\not = 0$ and for which either $\lambda_k(P_g)>\lambda_{k-1}(P_g)$ or $\lambda_k(P_g)<\lambda_{k+1}(P_g)$.
\begin{enumerate}[(i)]
\item If the metric $g$ is extremal for the normalized $k$-th eigenvalue functional, then there exists  a family of eigenfunctions $\{\phi_1,\cdots, \phi_p\}$ associated with $\lambda_k(P_g) $such that 
$\sum_{i=1}^p\tau_g(\phi_i) =0$ and $\sum_{i=1}^p\phi^2=\frac{2}{|\lambda_k(P_g)|}$.

\item If there exists a collection of eigenfunctions $\{\phi_1,\cdots, \phi_p\}$ associated with $\lambda_k(P_g)$ such that $\sum_{i=1}^p\phi_i^2$ is constant and $\sum_{i=1}^p\tau_g(\phi_i) = 0$,
then $g$ is extremal for the normalized $k$-th eigenvalue functional. 
\end{enumerate}
\end{theorem}

It is unclear what the meaning of the condition ``$\sum_{i=1}^p\tau_g(\phi_i)=0$" is in this case. Understanding such condition would be important to future developments in this theory.

\subsection{Organization of the paper} In section \ref{Preliminaries}, we argue that the one-sided derivatives of the eigenvalue functional (\ref{C-PFunctional3}) exist at $t=0$. This allows us to define what conformally extremal and extremal metrics are. In sections \ref{ProofMain1} and \ref{ProofMain2}, we discuss the proof of Theorems \ref{Main1} and \ref{Main2}. Some important consequences are discussed in section \ref{MainCorollaries}. In particular, it is shown that in certain cases where $\lambda_1(P_g)=0$, the first non-trivial eigenvalue admits no local minimum (Corollary \ref{NoLocalMinima}). We would like to point out that the theory developed in these sections is parallel to that of the Laplace-Beltrami operator, and some of the techniques are similar to those found in the work of El Soufi-Iias in \cite{Soufi, Soufi2}. In section \ref{ProofKorevaar}, we prove Theorem \ref{Main3} using ideas from \cite{Hassannezhad}. Section \ref{Sphere} is devoted to the case of the sphere. Finally, in section \ref{ExtremalMetrics}, we prove Theorem \ref{E-Main}.

\subsection{Acknowledgements} Part of this work was carried out during the author’s Ph.D. studies at the University of Notre Dame. The author would like to thank his doctoral adviser, Professor Matthew J. Gursky, for suggesting the problem and for explaining many important properties of the Paneitz operator on four-manifolds. This work was supported by the NSF grant RTG-DMS-1502424.

\section{Preliminaries: One-sided derivatives of the eigenvalue functional}\label{Preliminaries}

The main goal of this section is to define what conformally extremal metrics are. Let $g(t)=e^{2w_t}g_o$ be an analytic perturbation of an arbitrary metric $g_o$ in  $[g]$, i.e. $w_0 \equiv 0$ and $w_t$ depends real analytically in $t$ for $|t|$ small. We refer to the function $w_t$ as the \textit{generating function} for the perturbation $g(t)$. 

As in the case of linear operators on finite-dimensional vector spaces, the function $t\mapsto \lambda_k(P_{g(t)})$ is continuous but not differentiable in general. However, we will show that both $\frac{d}{dt}\lambda_k(P_{g(t)})|_{t=0^+}$ and $\frac{d}{dt}\lambda_k(P_{g(t)})|_{t=0^-}$ exists. We remark that the existence of the one-sided derivatives for Laplacian eigenvalues have been shown out in many works; see for instance \cite{Bando, Nadirashvili3, Soufi, Soufi2} and references therein. 

The existence of the one-sided derivatives rely on the following theorem from perturbation theory for linear operators. The original theory traces back to Rellich-Kato's work, but it was Canzani in \cite{Canzani} who proved that such a theory could be applied to a certain class of conformally covariant operators acting on smooth bundles over $M$, including GJMS operators like the Paneitz operator.  
 
\begin{proposition}[Rellich-Kato, Canzani]\label{RKC}
Let $\lambda_k(P_{g_o})$ be the $k$-th eigenvalue of the Paneitz operator with respect to $g_o\in[g]$, and denote by $m$ its multiplicity. Pick any analytic perturbation $g(t)=e^{2w_t}g_o$ of $g_o$. Then there exist $\Lambda_1(t),\cdots,\Lambda_m(t)$ analytic in $t$, and $\phi_1(t),\cdots,\phi_m(t)$ convergent power series in $t$ with respect to the $L^2$-norm topology, such that
\begin{equation}\label{RKC1}
P_{g(t)}\phi_i(t) = \Lambda_i(t)\phi_i(t), \text{ with } \Lambda_i(0)=\lambda_k(P_{g_o}), \;\; i=1,\cdots,m;
\end{equation}
and
\begin{equation}\label{RKC2}
\int_M \phi_i(t)\phi_j(t)\;dv_{g(t)} = \delta_{ij},\;\; i=1,\cdots,m.
\end{equation}
Moreover, if we select positive constants $d_1$ and $d_2$ such that the spectrum of $P_{g_o}$ in $(\lambda_k(P_{g_o}) -d_1,\lambda_k(P_{g_o}) + d_2)$ consists only of $\lambda_k(P_{g_o})$, then one can find a small enough $\delta >0$ such that the spectrum of $P_{g(t)}$ in $(\lambda_k(P_{g_o}) -d_1,\lambda_k(P_{g_o}) + d_2)$ consists of $\{\Lambda_1(t),\cdots,\Lambda_m(t)\}$ alone for $|t|<\delta$.
\end{proposition}


As a consequence, we can prove the key result of this section:

\begin{proposition}\label{OSD}
 Let $g(t)=e^{2w_t}g_o$ be any analytic perturbation of $g_o$, let $m$ be the multiplicity of $\lambda_k(P_{g_o})$, and denote by $\{\Lambda_i(t)\}_{i=1}^m$ and $\{\phi_i(t)\}_{i=1}^m$ the family arising from Proposition \ref{RKC}. The one-sided derivatives of $\lambda_k(P_{g(t)})$,
\[
\frac{d}{dt}\lambda_k(P_{g(t)})|_{t=0^+}\text{ and }\frac{d}{dt}\lambda_k(P_{g(t)})|_{t=0^-},
\]
both exist at $t=0$. Moreover, 
\begin{enumerate}[(i)]
\item If $\lambda_k(P_{g_o}) = \lambda_{k-j+1}(P_{g_o}) >\lambda_{k-j}(P_{g_o})$, then 
$\frac{d}{dt}\lambda_k(P_{g(t)})|_{t=0^+}$ and  $\frac{d}{dt}\lambda_k(P_{g(t)})|_{t=0^-}$ are the $j$-th smallest element and the $j$-th largest element in $\{\Lambda_1'(0),\cdots, \Lambda_m'(0)\}$, respectively.\footnote{To clarify: by the $j$-th smallest element we mean the $j$-th element in the collection when arranged in increasing order, while by the $j$-th largest element we mean the $j$-th element in the collection when arranged in decreasing order.}   

\item If $\lambda_k(P_{g_o})=\lambda_{k-j+1}(P_{g_o})<\lambda_{k+j}(P_{g_o})$, then
$\frac{d}{dt}\lambda_k(P_{g(t)})|_{t=0^+}$ and $\frac{d}{dt}\lambda_k(P_{g(t)})|_{t=0^-}$
are the $j$-th largest element and the $j$-th smallest element in $\{\Lambda_1'(0),\cdots, \Lambda_m'(0)\}$, respectively.
\end{enumerate}
In both cases,
\begin{equation}\label{AnalyticEigen}
\Lambda_i'(0) = -\lambda_k(P_{g_o})\int_M \alpha\phi_i(0)^2\;dv_{g_o},
\end{equation}
where $\alpha=4\frac{d}{dt}w_t|_{t=0}$.
\end{proposition}

\begin{proof}
 For the particular perturbation $g(t)$ of $g_o$, select $d>0$ such that the spectrum of $P_{g_o}$ in the interval $(\lambda_k(P_{g_o})-d, \lambda_k(P_{g_o})+d)$ consists only of $\lambda_k(P_{g_o})$, and then let $\delta>0$ as in Proposition \ref{RKC}. The continuity of $t\mapsto \lambda_k(P_{g(t)})$ and the analyticity of $t\mapsto \Lambda_i(t)$ imply that there exists $a,b\in\{1,\cdots,m\}$ and $\eta\in(0,\delta)$ such that
\begin{equation} \label{OSD1}
\lambda_k(P_{g(t)}) = \begin{cases}
			      \Lambda_a(t) & \hspace{.2in} \text{if } t\in[0,\eta) \\
			      \Lambda_b(t) & \hspace{.2in} \text{if } t\in(-\eta,0].
			      \end{cases}
\end{equation}
Notice that $a$ and $b$ could be different since $\Lambda_i(0) = \lambda_k(P_{g_o})$ for every $i=1,\cdots,m$. This shows the existence of the one-sided derivatives of $\lambda_k(P_{g(t)})$ at $t=0$.

We proceed with the proof of case $(i)$, and thus let us assume that $\lambda_k(P_{g_o}) = \lambda_{k-j+1}(P_{g_o}) > \lambda_{k-j}(P_{g_o})$. Noticed that our choice of $\delta$ implies that
\begin{equation}
\lambda_{k-j}(P_{g_o}) \not \in (\lambda_k(P_{g_o})-d,\lambda_k(P_{g_o})+d),
\end{equation}
and
\begin{equation}
\text{Spec}(P_{g(t)})\cap(\lambda_k(P_{g_o}) - d , \lambda_k(P_{g_o}) + d) = \{\Lambda_1(t),\cdots,\Lambda_m(t)\}
\end{equation}
for $|t|<\delta$. By taking $\delta$ smaller if necessary, we can assume that for any $i,j\in\{1,\cdots, m\}$ either the curves $\Lambda_i(t)$ and $\Lambda_j(t)$ coincide or they only intersect at $t=0$ for $|t|<\delta$. 

Since for $t\in[0,\delta)$ any pair of curves either coincide or only intersect at $t=0$, it make sense to arrange them in increasing order. Without loss of generality, after a possible re-index, we assume that $\Lambda_1(t)\le \cdots\le \Lambda_m(t)$ for $t\in[0,\delta)$. Therefore, $\lambda_k(P_{g(t)}) = \Lambda_j(t)$ for $t\in[0,\delta)$ and so
\[
\cdots\le \frac{\Lambda_{j-1}(t) - \lambda_k(P_{g_o})}{t} \le \frac{\lambda_k(P_{g(t)}) - \lambda_k(P_{g_o})}{t} \le \frac{\Lambda_{j+1}(t) - \lambda_k(P_{g_o})}{t}\le \cdots.
\]
This implies
\begin{equation}
\cdots\le \Lambda_{j-1}'(0)\le \frac{d}{dt}\lambda_k(P_{g(t)})|_{t=0^+} = \Lambda_j'(0)\le \Lambda_{j+1}'(0)\le \cdots,
\end{equation}
which proves that $\frac{d}{dt}\lambda_k(P_{g(t)})|_{t=0^+}$ is precisely the $j$-th smallest element in the collection $\{\Lambda_1'(0),\cdots, \Lambda_m'(0)\}$.

Similarly, for $t\in(-\delta,0]$, we can assume that $\Lambda_1(t)\le \cdots\le \Lambda_m(t)$ after a possible re-index, so that $\lambda_k(P_{g(t)}) = \Lambda_j(t)$ for $t\in (-\delta,0]$. Then 
\begin{equation}
\cdots\ge \frac{\Lambda_{j-1}(t) - \lambda_k(P_{g_o})}{t} \ge \frac{\lambda_k(P_{g(t)}) - \lambda_k(P_{g_o})}{t} \ge \frac{\Lambda_{j+1}(t) - \lambda_k(P_{g_o})}{t}\ge \cdots
\end{equation}
and thus 
\begin{equation}
\cdots\ge \Lambda_{j-1}'(0)\ge \frac{d}{dt}\lambda_k(P_{g(t)})|_{t=0^-} = \Lambda_j'(0)\ge \Lambda_{j+1}'(0)\ge \cdots,
\end{equation}
that is, $\frac{d}{dt}\lambda_k(P_{g(t)})|_{t=0^-}$ is the $j$-th largest element in $\{\Lambda_1'(0),\cdots, \Lambda_m'(0)\}$. The proof of case (ii) is similar and it is therefore omitted. 

It remains to show (\ref{AnalyticEigen}). We use the eigenvalue equation (\ref{RKC1}), together with the transformation law 
\begin{equation}\label{OSD-TransformationLaw}
P_{g_o} (\phi_i(t)) = \Lambda_i(t)e^{4w_t}\phi_i(t).
\end{equation}
Differentiating (\ref{OSD-TransformationLaw}) with respect to $t$ and setting $t=0$ gives 
\begin{equation}\label{OSD3}
P_{g_o}(\phi_i') = \Lambda_i'(0)\phi_i + \lambda_k(P_{g_o})\alpha\phi_i + \lambda_k(P_{g_o}) \phi_i',
\end{equation}
where $\phi_i:=\phi_i(0)$ and $\phi_i':= \frac{d}{dt}\phi_i(t)|_{t=0}$. On the other hand, observe that we obtain $P_{g_o} \phi_i = \lambda_k(P_{g_o})\phi_i$ by setting $t=0$ in (\ref{RKC1}), and thus
\begin{equation}\label{OSD4}
\phi_i'P_{g_o}\phi_i = \lambda_k(P_{g_o}) \phi_i'\phi_i.
\end{equation} 
Also, multiply (\ref{OSD3}) by $\phi_i$ to get
\begin{equation}\label{OSD5}
 \phi_iP_{g_o}\phi_i' = \Lambda_i'(0) \phi_i^2 + \lambda_k(P_{g_o})\alpha\phi_i^2 + \lambda_k(P_{g_o}) \phi_i\phi_i'. 
\end{equation}
Subtract (\ref{OSD4}) from (\ref{OSD5}) and integrate with respect to $dv_{g_o}$ to deduce
\begin{equation}
 \underbrace{\int_M(\phi_iP_{g_o}\phi_i' - \phi_i'P_{g_o}\phi_i)\;dv_{g_o}}_{=0\text{ by self-adjointness of }P_{g_o}} = \Lambda_i'(0)\underbrace{\int_M \phi_i^2\;dv_{g_o}}_{=1\text{ by }(\ref{RKC2})}+ \lambda_k(P_{g_o})\int_M\alpha\phi_i^2\;dv_{g_o}.
\end{equation}
Hence,
\begin{equation}
\Lambda_i'(0) = -\lambda_k(P_{g_o})\int_M \alpha\phi_i^2\;dv_{g_o}.
\end{equation}
\end{proof}

\begin{corollary}\label{OSD-cor}
Under the same assumptions as in Proposition \ref{OSD}, if $\lambda_k(P_{g_o})>\lambda_{k-1}(P_{g_o})$, then 
\begin{equation}\label{OSD+1}
\frac{d}{dt}\lambda_k(P_{g(t)})|_{t=0^+} = \min_{1\le i \le m} \Lambda'_i(0)  
\end{equation}
and
\begin{equation}\label{OSD-1}
\frac{d}{dt}\lambda_k(P_{g(t)})|_{t=0^-} = \max_{1\le i \le m} \Lambda'_i(0);
\end{equation}
while if $\lambda_k(P_{g_o})<\lambda_{k+1}(P_{g_o})$, then
\begin{equation}\label{OSD+2}
\frac{d}{dt}\lambda_k(P_{g(t)})|_{t=0^+} = \max_{1\le i \le m} \Lambda'_i(0)
\end{equation}
and
\begin{equation}\label{OSD-2}
\frac{d}{dt}\lambda_k(P_{g(t)})|_{t=0^-} = \min_{1\le i \le m} \Lambda'_i(0).
\end{equation}
\end{corollary}

\begin{remark}If we know the sign of $\lambda_k(P_{g_o})$, then the one-sided derivatives can be rewritten in terms of formula (\ref{AnalyticEigen}). For instance, if $\lambda_k(P_{g_o})>0$ and $\lambda_k(P_{g_o}) > \lambda_{k-1}(P_{g_o})$, so that we are in case (i) of Corollary \ref{OSD-cor}, then
\begin{equation}
\frac{d}{dt}\lambda_k(P_{g(t)})|_{t=0^+} = - \lambda_k(P_{g_o})\max_{1\le i \le m} \int_M \alpha\phi_i^2\;dv_{g_o},
\end{equation}
and
\begin{equation}
\frac{d}{dt}\lambda_k(P_{g(t)})|_{t=0^-} = - \lambda_k(P_{g_o}) \min_{1\le i\le m} \int_M \alpha\phi_i^2\;dv_{g_o}.
\end{equation}
\end{remark}

We end this section with the following definition. 

\begin{definition}\label{C-extremal}
A metric $g_w\in [g]$ is said to be \textit{conformally extremal} for $\lambda_k$ if and only if for any volume-preserving perturbation $g(t)\subset [g]$ of $g_w$ which is analytic in a neighborhood of $t=0$, we have 
\begin{equation}
\frac{d}{dt}\lambda_k(P_{g(t)})|_{t=0^+}\cdot \frac{d}{dt}\lambda_k(P_{g(t)})|_{t=0^-}\le 0.
\end{equation}
\end{definition}

\begin{remark} 
Suppose we have 
\begin{equation}\label{C-extremal2}
\frac{d}{dt}\lambda_k(P_{g(t)})_{t=0^-}\ge 0 \text{ and } \frac{d}{dt}\lambda_k(P_{g(t)})|_{t=0^+}\le 0
\end{equation}
As one can check, condition (\ref{C-extremal2}) is equivalent to 
\begin{equation}\label{C-extremal3}
\lambda_k(P_{g(t)})\le \lambda_k(P_{g_w}) + o(t) \text{ as }t\to 0.
\end{equation}
Similarly,
\begin{equation}\label{C-extremal4}
\frac{d}{dt}\lambda_k(P_{g(t)})_{t=0^-}\le 0 \text{ and } \frac{d}{dt}\lambda_k(P_{g(t)})|_{t=0^+}\ge 0
\end{equation}
is equivalent to 
\begin{equation}\label{C-extremal5}
\lambda_k(P_{g(t)})\ge \lambda_k(P_{g_w}) + o(t) \text{ as }t\to 0.
\end{equation}
Therefore, we could define conformally extremal by requiring that both (\ref{C-extremal3}) and (\ref{C-extremal5}) hold for any analytic perturbation $g(t)$ of $g_w$. This is analogous to the definition used by Nadirashvili in \cite{Nadirashvili3}.
\end{remark}

\begin{remark}
The usage of analytic approximation can be relaxed to smooth approximation, following the arguments of Fraser-Schoen in \cite{Fraser}. The main benefit of using analytic approximation is the differentiability of the family of eigenfunctions $\{\phi_i(t)\}_{i=1}^{p}$ in Proposition \ref{RKC}. This simplifies some of the arguments and it is enough for the purposes of our work.
\end{remark}

\section{Conformally Extremal Metrics}

In this section, we prove the main theorems concerning conformally extremal metrics. As mentioned in the introduction, theorems \ref{Main1}, \ref{Main2}, and \ref{Main3} generalize known results in the case of Laplacian eigenvalues on surfaces. Let us briefly discuss some of these here. 

Recall that on closed surfaces there is a topological upper bound (see \cite{Yang, Korevaar, Karpukhin3}) for each normalized eigenvalue $\lambda_k(-\Delta_g)\text{Area}(M^2,g)$ as $g$ varies over the space $\mathcal{R}(M^2)$ of all Riemannian metrics. In particular, each normalized eigenvalue functional is bounded when restricted to a conformal class. Some of the main results concerning conformally extremal metrics are:

\begin{enumerate}
\item\label{1} If $g_w\in [g]$ is conformally extremal for the normalized eigenvalue functional $g_w\in [g]\mapsto \lambda_k(-\Delta_{g_w})\text{Area}(M^2,g_w)$\footnote{Recall that in order to stay consistent with our notation, we denote by $\lambda_k(-\Delta_g)$ the actual $k$-th eigenvalue of $-\Delta_g$.} ($k\ge 2$), then there exists a family of eigenfunctions $\{\phi_1,\cdots, \phi_p\}$ associated with $\lambda_k(-\Delta_g)$ such that $\sum_{i=1}^p\phi_i^2\equiv 1$ on $M^2$. By differentiating this constraint, we find that $\lambda_k(-\Delta_g)= \sum_{i=1}^p|\nabla_g\phi|^2$, and thus the map $\Phi = (\phi_1,\cdots,\phi_p)$ defines a harmonic map with constant energy density into the sphere $\mathbb{S}^{p-1}$ (see for instance \cite{Soufi, Soufi2}).

\item\label{2} If either $\lambda_k(-\Delta_g)>\lambda_{k-1}(-\Delta_g)$ or $\lambda_k(-\Delta_g)<\lambda_{k+1}(-\Delta_g)$, then the necessary conditions in (1) are also sufficient for the existence of conformally extremal metrics (Theorem 4.1 in \cite{Soufi2}). That is, if $(M^2,g)$ admits a harmonic map with constant density into a sphere, then $g$ is conformally extremal for some eigenvalue. 
\end{enumerate}

We remark that the existence of maximal metrics in conformal classes has been proved in the case of closed surfaces (\cite{Petrides1,Petrides2,Karpukhin}). It is natural to study similar existence questions for eigenvalues of the Paneitz operator in conformal classes. Notice that there are conformal classes for which the Paneitz operator have negative eigenvalues; see \cite{Canzani2} for explicit examples. Moreover, the number of negative eigenvalues $N([g])$ is a conformal invariant (\cite{Canzani2}). 
Therefore, some of the techniques developed in \cite{Gursky3} by Gursky and the author for maximizing negatives eigenvalues of the Conformal Laplacian in higher dimensional manifolds could work in this setting. 

\subsection{Proof of Theorem \ref{Main1}}\label{ProofMain1}

Theorem \ref{Main1}, and its consequence Corollary \ref{Cor-Main1}, generalize (\ref{1}) above to the Paneitz operator on closed four-manifolds. Before starting with the proof of Theorem \ref{Main1}, we discuss the following lemma:

\begin{lemma}[Leibniz rule for $P_g$]\label{P-LR} For smooth functions $\phi$ and $\psi$ on $M^4$, we have
\begin{equation}\label{P-LR1}
\begin{split}
    P_g(\phi\psi) =& \psi P_g(\phi) + \phi P_g(\psi) + 2(\Delta_g\phi)(\Delta_g \psi) + 2\langle \nabla_g\Delta_g\phi,\nabla_g\psi\rangle \\ &+ 2\langle \nabla_g\Delta_g\psi, \nabla_g \phi\rangle + 2\Delta_g\langle \nabla_g\phi,\nabla_g\psi\rangle - \frac{4}{3}R_g\langle\nabla_g\phi,\nabla_g \psi\rangle  \\&+4\text{Ric}_g(\nabla_g\phi,\nabla_g\psi)
\end{split}
\end{equation}
In particular,
\begin{equation}\label{P-LR2}
\begin{split}
    P_g(\phi^2) =& 2\phi P_g(\phi) + 2(\Delta_g\phi)^2 + 4 \langle \nabla_g\Delta_g\phi,\nabla_g\phi\rangle + 2\Delta_g|\nabla_g\phi|^2 \\&- \frac{4}{3}R_g|\nabla_g\phi|^2 + 4\text{Ric}_g(\nabla_g\phi,\nabla_g\phi) 
\end{split}
\end{equation}
\end{lemma}
\begin{proof}
It is enough to check (\ref{P-LR1}) locally. We start by computing $\Delta_g^2(\phi\psi)$:
\begin{equation}
\begin{split}\label{P-LR3}
    \Delta_g^2(\phi\psi) =& \Delta_g(\phi\Delta_g\psi + \psi\Delta_g\phi + 2\langle\nabla_g\phi,\nabla_g\psi\rangle) \\ =&\psi \Delta_g^2\phi + \phi\Delta_g^2\psi + 2(\Delta_g\phi)(\Delta_g\psi) + 2\langle \nabla_g\Delta_g\phi,\nabla_g\psi\rangle \\ &+2\langle \nabla_g\Delta_g\psi, \nabla_g \phi\rangle + 2\Delta_g\langle \nabla_g\phi,\nabla_g\psi\rangle 
\end{split}
\end{equation}
From (\ref{P-local}), we then deduce 
\begin{equation}
\begin{split}
    P_g(\phi\psi) =& \Delta^2_g(\phi\psi) - \frac{2}{3}\text{div}_g(R_g\nabla_g(\phi\psi)) + 2\sum_{i=1}^4\text{div}_g(\text{Ric}_g(\nabla_g(\phi\psi),E_i)E_i)\\ =& \Delta_g^2(\phi\psi) -\frac{2}{3}\text{div}_g(\psi R_g\nabla_g\phi) -\frac{2}{3}\text{div}_g(\phi R_g\nabla_g\psi) \\ &+2\sum_{i=1}^4\text{div}_g(\psi\text{Ric}_g(\nabla_g\phi,E_i)E_i) +2\sum_{i=1}^4\text{div}_g(\phi\text{Ric}_g(\nabla_g\psi,E_i)E_i)\\ =& \psi P_g(\phi) + \phi P_g(\psi) + 2(\Delta_g\phi)(\Delta_g \psi) + 2\langle \nabla_g\Delta_g\phi,\nabla_g\psi\rangle \\ &+ 2\langle \nabla_g\Delta_g\psi, \nabla_g \phi\rangle + 2\Delta_g\langle \nabla_g\phi,\nabla_g\psi\rangle - \frac{4}{3}R_g\langle\nabla_g\phi,\nabla_g \psi\rangle  \\&+4\text{Ric}_g(\nabla_g\phi,\nabla_g\psi).
\end{split}
\end{equation}
To obtain (\ref{P-LR2}) we simply take $\psi=\phi$. This completes the proof.
\end{proof}

We can now proceed with the proof of Theorem \ref{Main1}. The main tool we use is Hahn-Banach classical separation theorem; see Chapter 1 in \cite{Brezis}. The proof is similar to the one for laplacian eigenvalues as discussed in \cite{Soufi, Soufi2, Petrides1, Fraser}.

\begin{proof}[Proof of Theorem \ref{Main1}]
Consider the subset 
\begin{equation}\label{subset1}
\left\{\phi^2:\; \phi\in E_k(P_{g_e})\text{ and }\|\phi\|_{L^2(M,g_e)}=1\right\} 
\end{equation}
of $L^2(M^4,g_e)$, where $E_k(P_{g_e})$ denotes the $k$-th eigenspace of $P_{g_e}$, and denote by $K$ its convex hull. $K$ is the set of all finite convex combinations of the initial set subset:
\[
K= \left\{\sum_{\text{finite}} a_j\phi_j^2:\;a_i\ge 0,\; \sum a_i=1, \phi_i\in E_k(P_{g_e}), \|\phi_i\|_{L^2(M,g_e)}=1 \right\}.
\]
The proof would be completed if we can show that $1\in K$. Since $E_k(P_{g_e})$ is a finite-dimensional subspace, we deduce that the initial subset lies in a finite-dimensional subspace. By Caratheodory's Theorem for Convex Hulls, the compactness of the original set then imply the compactness of $K$. 

We argue by contradiction, that is, assume $1\not \in K$. Then $\{1\}$ and $K$ are two disjoint convex sets, $\{1\}$ being closed, and $K$ being compact. Hahn Banach Separation Theorem implies that we can find a function $f$ in $L^2(M^4,{g_e})$ that separates $\{1\}$ and $K$:
\begin{equation}\label{Main1-eq1}
\int_M f\;dv_{g_e} >0,
\end{equation}
and
\begin{equation}\label{Main1-eq2}
\int_M f\varphi\;dv_{g_e} \le 0
\end{equation}
for all $\varphi\in K$. By standard approximation arguments, we can assume that $f$ is smooth.

Let $\tilde{f} = f - \text{Vol}(M^4,g_e)^{-1}\displaystyle\int_M f\;dv_{g_e}$, and consider the conformal perturbation of $g_e$ given by 
\begin{equation}\label{choice-perturbation}
g(t)= \frac{\text{Vol}(M^4,g_e)^{\frac{1}{2}}}{\text{Vol}(M^4,e^{t\tilde f}g_e)^{\frac{1}{2}}} e^{t\tilde{f}}g_e.
\end{equation}
This is an analytic, volume-preserving and conformal perturbation of $g_e$. Moreover, since $\frac{d}{dt} \text{Vol}(M^4,e^{t\tilde f}g_e) |_{t=0}= 2\displaystyle\int_M \tilde{f}\;dv_{g_e} = 0$ by our choice of $\tilde{f}$, we obtain $\frac{d}{dt} g(t)|_{t=0} = \tilde{f} {g_e}$. Following the notation of Proposition \ref{OSD}, set $\alpha=2\tilde{f}$ and write $\frac{d}{dt} g(t)|_{t=0} = \frac{1}{2}\alpha g_e$. 

Now, $g_e$ being conformally extremal means that the one-sided derivatives of $\lambda_k(P_{g(t)})$ along this particular perturbation should have different signs. However, by (\ref{Main1-eq1}) and (\ref{Main1-eq2}), for any $\phi\in E_k(P_{g_e})$ we find that the quantity
\begin{equation}
\begin{split}
\int_M \alpha \phi^2\;dv_{g_e} &= 2 \int_M \tilde{f}\phi^2\;dv_{g_e}\\ &= 2\underbrace{\int_M f\phi^2\;dv_{g_e}}_{\le 0} - 2\text{Vol}(M^4,g_e)^{-1}\underbrace{\left(\int_M f\;dv_{g_e}\right)}_{>0}\underbrace{\left(\int_M \phi^2\;dv_{g_e}\right)}_{>0}
\end{split}
\end{equation}
has a constant sign on $E_k(P_{g_e})$. Going back to (\ref{OSD-1}) in Corollary \ref{OSD-cor}, we deduce that this contradicts our assumption of $g_e$ being extremal. Hence, $1\in \text{Conv}(K)$.

We proceed now with the proof of (\ref{Main1Formula}). To this end, we apply Lemma \ref{P-LR} to compute $P_{g_e}\left(\sum_{i=1}^p\phi^2\right)$:
\[
\begin{split}
0 =& P_{g_e}(1) = P_{g_e}\left(\sum_{i=1}^p\phi_i^2\right) = \sum_{i=1}^p P_{g_e}(\phi_i^2) \\ =& \sum_{i=1}^p\Bigg{\{}2\phi_i^2\lambda_k(P_{g_e}) + 2(\Delta_{g_e}\phi_i)^2 + 4\langle\nabla_{g_e}\phi_i,\nabla_{g_e}\Delta_{g_e}\phi_i\rangle_{g_e} + 2\Delta_{g_e}|\nabla_{g_e}\phi_i|^2_{g_e} \\ &\hspace{.4in} - \frac{4}{3}R_{g_e}|\nabla_{g_e}\phi_i|_{g_e}^2 +4\text{Ric}_{g_e}(\nabla_{g_e}\phi_i,\nabla_{g_e}\phi_i)\Bigg{\}}.
\end{split}
\]
Since $\sum_{i=1}^p\phi_i^2 = 1$, after dividing both sides by $2$ we obtain
\begin{equation}\label{Main1-line1}
\begin{split}
\lambda_k(P_{g_e}) =& -\sum_{i=1}^p \Bigg{\{}(\Delta_{g_e}\phi_i)^2 + 2\langle\nabla_{g_e}\phi_i,\nabla_{g_e}\Delta_{g_e}\phi_i\rangle_{g_e} + \Delta_{g_e}|\nabla_{g_e}\phi_i|_{g_e}^2\\ &\hspace{.5in} - \frac{2}{3}R_{g_e}|\nabla_{g_e}\phi_i|_{g_e}^2 + 2\text{Ric}_{g_e}(\nabla_{g_e}\phi_i,\nabla_{g_e}\phi_i)\Bigg{\}}
\end{split}
\end{equation}
On the other hand, by applying formula (\ref{P-LR3}) to $\sum_{i=1}^p \phi_i^2 = 1$, we deduce
\begin{equation}\label{Main1-line2}
\begin{split}
0 &= \frac{1}{2}\Delta^2_{g_e}\left(\sum_{i=1}^p\phi_i^2\right) \\ &= \sum_{i=1}^p\left\{\phi_i\Delta^2_{g_e}\phi_i + (\Delta_{g_e}\phi_i)^2 + 2\langle\nabla_{g_e} \phi_i,\nabla_{g_e}\Delta_{g_e}\phi_i \rangle_{g_e} + \Delta_{g_e}|\nabla_{g_e} \phi_i|_{g_e}^2\right\}.
\end{split}
\end{equation}
The desired result follows after substituting (\ref{Main1-line2}) into (\ref{Main1-line1}).

Finally, we show that $\lambda_k(P_{g_e})$ is degenerate. If the multiplicity of $\lambda_k(P_{g_e})$ is $1$, then the extremal condition becomes $\phi^2\equiv 1$ for some eigenfunction associated with $\lambda_k(P_{g_e})$. This gives us that $\phi$ is constant, which is a contradiction since $\lambda_k(P_{g_e})\not = 0$ by assumption. This concludes the proof of the theorem.
\end{proof}

\subsection{Proof of Theorem \ref{Main2}}\label{ProofMain2}

Theorem \ref{Main2}, part (i), generalizes (\ref{2}) to the Paneitz operator on closed four-manifolds. Part (ii) in Theorem \ref{Main2} generalizes the following result in the case of conformally extremal metrics for the Laplace-Beltrami operator: if $(M^2,g)$ admits a non-degenerate harmonic map into some sphere $(\mathbb{S}^{p-1},g_r)$, then there is a metric $g_{w}\in[g]$ which is conformally extremal for $\lambda_k$ for some $k$; see Theorem 2.3 in \cite{Karpukhin4}. In our case, we have to account for the possibility of negative eigenvalues. 

\begin{proof}[Proof of Theorem \ref{Main2}, part (i)]
Without loss of generality, we assume that $\sum_{i=1}^p\phi^2=1$. Multiplying both sides by an arbitrary test function $\alpha$ and integrating gives
\begin{equation}\label{Main2-1}
\sum_{i=1}^p\int_M \alpha\phi_i^2\;dv_g = \int_M \alpha dv_{g_e}. 
\end{equation}
Now, let $g(t) = e^{2w_t}g_e$ be any analytic, volume-preserving perturbation of $g_e$ in $[g]$. By setting $\alpha=4\frac{d}{dt}w_t|_{t=0}$, the volume constraint gives
\begin{equation}\label{Main2-2}
0 = \int_M \left(4\frac{d}{dt}w_t|_{t=0}\right)\;dv_{g_e} = \int_M \alpha dv_{g_e}.
\end{equation}
Therefore, for any analytic, volume-preserving perturbation $g(t)$ of $g_e$ in $[g]$, it holds that
\begin{equation}\label{Main2-3}
\sum_{i=1}^p\int_M \alpha\phi_i^2\;dv_g = \int_M \alpha dv_{g_e}=0.
\end{equation}
From Corollary \ref{OSD-cor}, we conclude that the metric $g_e$ is conformally extremal.
\end{proof}

\begin{proof}[Proof of Theorem \ref{Main2}, part (ii)]. 
Recall that $\Phi=(\phi_1,\cdots, \phi_p)$ solves the equation $P_g(\Phi) = e_g(\Phi)\Phi$. Assume $e_g(\Phi)>0$ pointwise on $M^4$. We take $g_\Phi$ to be the conformal metric $g_\Phi = e_g(\Phi)^{\frac{1}{2}}g$. Then
\begin{equation}
P_{g_\Phi}(\phi_i) = \phi_i,
\end{equation}
which means that the components functions of $\Phi$ are eigenfunctions of $P_{g_\Phi}$ with corresponding eigenvalue $\lambda_k(P_{g_\Phi}) = 1$, where we have chosen $k$ such that $\lambda_k(P_{g_\Phi}) > \lambda_{k-1}(P_{g_\Phi})$. Moreover, since $\sum_{i=1}^p\phi_i^2=1$, by expanding $0=P_{g_\phi}(\sum_{i=1}^p\phi_i^2)$ as in the proof of Theorem \ref{Main1}, we conclude that $e_{g_\Phi}(\Phi)=\lambda_k(P_{g_\Phi})$ is constant. The result now follows from Theorem \ref{Main2}, part (i).

In the case where $e_g(\Phi)<0$ pointwise on $M^4$, we take $g_{\Phi} = (-e_g(\Phi))^{\frac{1}{2}}g$. Then the component functions of $\Phi$ are eigenfunctions of $P_{g_\Phi}$ with corresponding eigenvalue $-1$. The rest of the proof is as in the previous case. 
\end{proof}

\subsection{Corollaries of Theorems \ref{Main1} and \ref{Main2}}\label{MainCorollaries}

The following result is due to Canzani in \cite{Canzani}. It is a generalization to a larger class of elliptic operators of classical results about generic properties for laplacian eigenvalues (cf. \cite{Bando, Uhlenbeck}). For simplicity, we state the result only for the Paneitz operator. 
\begin{theorem}[Theorem 2.1 in \cite{Canzani}]
Let $F$ be the subset of $C^\infty(M^4,\mathbb{R})$ consisting of those smooth functions $w$ for which all non-zero eigenvalues of $P_{e^{2w}g}$ are simple:
$F:=\{w\in C^\infty(M^4,\mathbb{R}): \text{dim}(E_k(P_{e^{2w}g}))=1\; \text{for all}\; k\in \mathbb{N}\;\text{for which}\; \lambda_k(P_{e^{2w}g})\not=0 \}$. Then the set $F$ is residual, meaning that $F$ equals a countable intersection of open and dense subsets of $C^\infty(M^4,\mathbb{R})$.
\end{theorem}

Since eigenvalues associated with conformally extremal metrics are degenerate, as a consequence of Theorem \ref{Main1}, the set of smooth functions $w$ for which $\lambda_k(P_{e^{2w}g})$ is conformally extremal for some $k$ is a subset of a meagre set. In fewer words, most Riemannian metrics $g_w\in [g]$ cannot be conformally extremal for any eigenvalue.

By the last part of Theorem \ref{Main1}, if $g_e\in [g]$ is conformally extremal for the normalized $k$-th eigenvalue functional and $\lambda_k(P_{g_e})\not=0$, then either $\lambda_k(P_{g_e})=\lambda_{k-1}(P_{g_e})$ or $\lambda_k(P_{g_e}) = \lambda_{k+1}(P_{g_e})$. Depending on whether $g_e$ is a local maximizer or a local minimizer, more could be established. 

\begin{proposition}\label{LocalExtremals}
Let $(M^4,g)$ be a four-dimensional closed Riemannian manifold.
\begin{enumerate}[(i)]
\item If $g_e$ is a local maximizer for (\ref{C-PFunctional}) and $\lambda_k(P_{g_e})\not = 0$, then $\lambda_k(P_{g_e}) = \lambda_{k+1}(P_{g_e})$.
\item If $g_e$ is a local minimizer for (\ref{C-PFunctional}) and $\lambda_k(P_{g_e})\not = 0$, then $\lambda_k(P_{g_e})=\lambda_{k-1}(P_{g_e})$.
\end{enumerate}
\end{proposition}

\begin{proof}
Denote by $A$ the space of all constant functions. Then we can write $L^2(M,g_e)$ $= A\oplus A^\perp$, where the orthogonal sum is with respect to the $L^2$-inner product induced by $dv_{g_e}$. Notice that $A^\perp$ is just the space of zero mean value functions with respect to $dv_{g_e}$.

The proof is by contradiction. Assume that $g_e$ is a local maximizer for (\ref{C-PFunctional3}), and that $\lambda_k(P_{g_e})<\lambda_{k+1}(P_{g_e})$. For any function $\varphi\in A^\perp$, select an analytic perturbation $g(t)=e^{2w_t}g_e\in [g]$ satisfying 
\begin{equation}
\frac{1}{2}\alpha:= 2\frac{d}{dt}w_t|_{t=0}=\varphi.
\end{equation}
Since $\displaystyle \int_M\varphi \;dv_{g_e} = 0$, the perturbation $g(t)$ keeps the volume fixed. If $m$ denotes the multiplicity of $\lambda_k(P_{g_e})$, then by applying Proposition \ref{RKC} we deduce the existence of families $\Lambda_1(t),\cdots, \Lambda_m(t)$ and $\phi_1(t),\cdots,\phi_m(t)$ such that 
\begin{equation}
P_{g(t)}\phi_i(t) = \Lambda_i(t)\phi_i(t), \text{ with } \Lambda_i(0)=\lambda_k(P_{g_e}) \text{ for all } i=1,\cdots,m;
\end{equation}
and
\begin{equation}
\int_M \phi_i(t)\phi_j(t)\;dv_{g(t)} = \delta_{ij} \text{ for all } i=1,\cdots,m.
\end{equation}
Moreover, there exists a number $\delta>0$ such that the spectrum of $P_{g(t)}$  around $\lambda_k(P_{g_e})$ consists only of $\{\Lambda_1(t),\cdots,\Lambda_m(t)\}$ for $|t|<\delta$. Since $\lambda_k(P_{g_e}) < \lambda_{k+1}(P_{g_e})$, and both $\lambda_k(P_{g(t)})$ and $\Lambda_i(t)$ are continuous, we have 
\begin{equation}\label{LE-1}
\lambda_k(P_{g(t)})=\max_{1\le i\le m} \Lambda_i(t)
\end{equation}
on $(-\delta, \delta)$. Together with the assumption of $g_e$ being a local maximizer, relation (\ref{LE-1}) gives 
\begin{equation}
\Lambda_i(t)\le \lambda_k(P_{g(t)})\le \lambda_k(P_{g_e}) = \Lambda_i(0)
\end{equation}
for $|t|<\delta'$, where $\delta'$ is possibly smaller than our original $\delta$. This implies 
\begin{equation}\label{LE-2}
    \Lambda_i'(0)=0
\end{equation}
for $i=1,\cdots,m$.

Set $\phi_i= \phi_i(0)$, and recall that the collection of eigenfunctions $\{\phi_1,\cdots,\phi_m\}$ form an orthonormal basis for the $k$-th eigenspace $E_k(P_{g_e})$. From (\ref{AnalyticEigen}) and (\ref{LE-2}) we get
\begin{equation}
\int_M \varphi \phi_i^2\;dv_{g_e}=0.
\end{equation}
Moreover, after multiplying (\ref{OSD3}) across by $\phi_j$ and integrating with respect to $dv_{g_e}$ (see proof of Proposition \ref{OSD}), we obtain
\begin{equation}
\int_M \varphi \phi_i \phi_j\;dv_{g_e} = 0 \text{   (for }i\not = j).
\end{equation}
Therefore, $\displaystyle\int_M \varphi \phi^2\;dv_{g_e} = 0$ for all $\phi\in E_k(P_{g_e})$. Since $\varphi\in A^\perp$ was arbitrary, this gives that $\phi^2=c\in \mathbb{R}$ for every $\phi\in E_k(P_g)$. However, this is a contradiction since $\lambda_k(P_{g_e})\not=0$. The proof for $(ii)$ is similar and it is therefore omitted. 
\end{proof}

\begin{corollary}\label{NoLocalMinima}
Let $M^4$ be a four-dimensional closed Riemannian manifold equipped with a conformal class for which Ker$(P_g) = \{\text{constants}\}$ and $P_g\ge 0$. Then the functional 
\begin{equation}\label{first-nonzero-functional}
g_w\in [g] \mapsto \lambda_2(g_w)\text{Vol}(M^4,g_w)
\end{equation}
has no local minimizer.
\end{corollary}
\begin{proof}
Assume on the contrary that $g_e\in [g]$ is a local minimizer for the first nonzero eigenvalue functional (\ref{first-nonzero-functional}). Then Proposition \ref{LocalExtremals} implies that $\lambda_2(P_{g_e}) = \lambda_1(P_{g_e})=0$, which contradicts our assumptions on $\text{Ker}(P_g)$. 
\end{proof}

\begin{remark}
Notice that both of our assumptions, $\text{Ker}(P_g)=\{\text{constants}\}$ and $P_g\ge 0$, are conformally invariant. Moreover, there are plenty of conformal classes satisfying these conditions. In fact, by the work of Gursky in \cite{Gursky}, if both the total Q-curvature $k([g])$ and the Yamabe invariant $Y([g])$ are nonnegative, then both conditions are satisfied by the conformal class; see \cite{Gursky2} for a stronger version of this result. 
\end{remark}

As an application of Theorem \ref{Main2}, part (i), we provide  examples of Riemannian manifolds with conformally extremal metrics on it (cf. Corollary 2.1 in \cite{Soufi}). As we mentioned in the introduction, section \ref{PaneitzIntro}, these conformally extremal metrics may or may not be global maximizers, and classifying them is another important question. 

\begin{proof}[Proof of Corollary \ref{Homogeneous}]
Let us start by choosing an orthonormal basis $\{\phi_1,\cdots,\phi_p\}$ of $E_k(P_g)$. In light of Theorem \ref{Main2}, our goal is to show that $\sum_{i=1}^p\phi_i^2$ is constant. To this end, let $G$ be a Lie group acting smoothly and transitively by isometries on $M^4$. Note that for any $\varphi \in G$, the set $\{\varphi^*\phi_1,\cdots, \varphi^*\phi_p\}$ is also an orthonormal basis for $E_k(P_g)$. For an arbitrary test function $\psi\in C^\infty(M^4)$, since the trace of the quadratic form $Q_\psi(\phi):=\displaystyle\int_M\psi \phi^2dv_g$ on $E_k(P_g)$ is basis independent, we deduce that
\begin{equation}
\int_M \psi \left(\sum_{i=1}^p \phi_i^2\right)\;dv_g = \int_M\psi \left(\sum_{i=1}^p(\varphi^*\phi_i)^2\right)\;dv_g.
\end{equation}
This implies that the function $\sum_{i=1}^p\phi_i^2$ is invariant under the isometry group $G$. Hence, it is constant and the proof is completed.
\end{proof}

\subsection{Proof of Theorem \ref{Main3}}\label{ProofKorevaar}

We define the \textit{conformal spectrum} of the Paneitz operator similar to how it is defined on closed surfaces for the Laplace-Beltrami operator (see \cite{Colbois}). 
\begin{definition}[Conformal Spectrum]
For any $k\in \mathbb{N}$, we define the $k$\textit{-th conformal eigenvalue} of $P_g$ to be
\begin{equation}
\Lambda^P_k(M^4,[g]):=\sup_{g_w\in\; C(g)} \lambda_k(P_{g_w}) =  \sup_{g_w\in [g]}\lambda_k(P_{g_w})\text{Vol}(M^4,g_w).
\end{equation}
The sequence 
\begin{equation}
\Lambda^P_1(M^4,[g])\le \Lambda^P_2(M^4,[g])\le\cdots \le \Lambda^P_k(M^4,[g])\rightarrow \infty
\end{equation}
is called the \textit{Conformal Spectrum} of $P_g$. Notice that, unlike in the case of surfaces, some conformal eigenvalues could be negative. 
\end{definition}

For conformal eigenvalues of the Laplace-Beltrami operator, it was shown by Colbois-El Soufi in \cite{Colbois} that no metric can maximize two consecutive normalized eigenvalues. This follows from Theorem B in \cite{Colbois} and Proposition 4.2 in \cite{Soufi2}. As a consequence of Proposition \ref{LocalExtremals}, we managed to prove a weaker version of this result.

\begin{corollary}\label{Consecutive}
Let $(M^4,g)$ be a four-dimensional closed Riemannian manifold. Assume that $g_o\in [g]$ is a unit volume metric that maximizes the $k$-th conformal eigenvalue, i.e.
\begin{equation}
\lambda_k(P_{g_o}) = \Lambda^P_k(M^4,[g]),
\end{equation}
and suppose that the multiplicity of $\lambda_k(P_{g_o})$ is $m$. Without loss of generality, if $k\ge 2$, then we further assume that $\lambda_{k-1}(P_{g_o}) < \lambda_{k}(P_{g_o})$. Then $g_o$ cannot maximize at least one conformal eigenvalue in the following list:
\begin{equation}
\Lambda^P_{k+1}(M^4,[g]),\cdots, \Lambda^P_{k+m-1}(M^4,[g]).
\end{equation}
In other words, $g_o$ cannot maximize more than $m-1$ consecutive eigenvalues. \end{corollary}

\begin{proof}
We first explain one of our assumptions. Notice that if $l<k$ and $\lambda_l(P_{g_o}) = \lambda_k(P_{g_o})$, then 
\begin{equation}
    \lambda_l(P_{g_o})=\Lambda^P_l(M^4,[g]).
\end{equation}
Therefore, we are simply assuming that $\Lambda^P_k(M^4,[g])$ is the first one appearing on the list of possibly repeated conformal eigenvalues.

If $g_o$ does not maximize at least one of the conformal eigenvalues in 
\begin{equation}
\Lambda^P_{k+1}(M^4,[g]),\cdots, \Lambda^P_{k+m-2}(M^4,[g]),
\end{equation}
then we are done. Otherwise, the proof reduces to show that 
\begin{equation}\label{Consecutive-1}
\lambda_{k+(m-1)}(P_{g_o}) < \Lambda^P_{k+(m-1)}(M^4,[g]).
\end{equation}
Indeed, if equality holds in (\ref{Consecutive-1}), then $g_o$ would be a global maximizer. Therefore, $\lambda_{k+(m-1)}(P_{g_o}) = \lambda_{k+m}(P_{g_o})$ by Proposition \ref{LocalExtremals}, contradicting our assumption on the multiplicity of $\lambda_k(P_{g_o})$. This completes the proof.
\end{proof}

We would like to explain how Corollary \ref{Consecutive} could be improved. For simplicity, let us assume that we are given a conformal class $[g]$ for which $P_g$ is nonnegative ($N([g])=0$) and $\text{Ker}(P_g) = \{\text{constants}\}$. As explained in the introduction, section \ref{PaneitzIntro}, these conditions are conformally invariant. Suppose that we can find a uniform constant $A>0$ such that
\begin{equation}\label{Consecutive-2}
    \Lambda^P_{k+1}(M^4,[g])^a - \Lambda^P_k(M^4,[g])^a \ge A
\end{equation}
for some $a>0$. Then, under the assumptions of Corollary \ref{Consecutive}, we would be able to conclude that $g_o$ does not maximize $\Lambda^P_{k+1}(M^4,[g])$. In other words, there would be no metric maximizing consecutive conformal eigenvalues. Indeed, if $g_o$ maximizes $\Lambda^P_k(M^4,[g])$, then Corollary \ref{LocalExtremals} implies $\Lambda^P_k(M^4,[g]) = \lambda_k(P_{g_o})=\lambda_{k+1}(P_{g_o})$. Since by (\ref{Consecutive-2}) there would be a positive gap between any two consecutive conformal eigenvalues, we would therefore conclude that $\lambda_{k+1}(P_{g_o}) < \Lambda^P_{k+1}(M^4,[g])$. 

An inequality of the form (\ref{Consecutive-2}) is known to be true in the case of the Laplace-Beltrami operator with $a=\frac{\text{dim}(M^n)}{2}$ and $A$ equal to the first non-trivial conformal eigenvalue on the round sphere; see Theorem B in \cite{Colbois} for more details.

Our next goal is to prove Theorem \ref{Main3}. Recall that the min-max characterization for the eigenvalues of $P_{g_w}$ is given by
\begin{equation}\label{min-maxP}
\lambda_k(P_{g_w})= \inf_{V\in Gr_{k}(C^\infty(M^4))} \sup_{\phi \in V\setminus \{0\}} R_{g_w}(\phi),
\end{equation}
where $Gr_k(C^\infty(M^4))$ is the space of all $k$-dimensional subspaces of $C^\infty(M^4)$, and $R_{g_w}(\phi)$ is the associated Rayleigh quotient defined by 
\begin{equation}\label{RQ-P1}
R_{g_w}(\phi) = \frac{\displaystyle\int_M \phi P_{g_w} \phi \;dv_{g_w}}{\displaystyle \int_M \phi^2\;dv_{g_w}} = \frac{\displaystyle\int_M\phi P_g\phi\;dv_g}{\displaystyle\int_M \phi^2 \;dv_{g_w}}.
\end{equation}

The terminology introduced in the next definition is needed to understand Hassannezhad's result, Lemma \ref{Asma} below.

\begin{definition}
Let $(X,d,\mu)$ be a complete and locally compact metric-measure space with a metric d and a non-atomic measure $\mu$.
\begin{enumerate}[(i)]
\item (Local covering property) Given $\kappa>1$, $\rho>0$ and $N\in\mathbb{N}$, we say that $(X,d)$ satisfies the \textit{$(\kappa,N;\rho)$-covering property} if each ball of radius $0<r\le\rho$ can be covered by at most $N$ balls of radius $\frac{r}{\kappa}$.
\item For any $x\in X$ and $0\le r\le R$, we define the \textit{annulus} $A(x,r,R)$ as the set
\begin{equation}
A(x,r,R):=\{y\in X: r<d(x,y)\le R\}
\end{equation}
The set $2A(x,r,R)$ is defined to be $A\left(x,\frac{r}{2},2R\right)$.
\end{enumerate}
\end{definition}

\begin{lemma}[Proposition 2.1 in \cite{Hassannezhad}]\label{Asma}
Let $(X,d,\mu)$ be a metric-measure space satisfying the $(2,N; 1)$-covering property, with $X$ compact and $\mu$ a non-atomic finite measure. Then there exists a constant $k_X$, depending on the covering property, such that for all $n>k_X$ there exists a family of annuli $\mathcal{A} = \{A_i\}_{i=1}^n$ with the following properties:
\begin{enumerate}[(i)]
\item $\mu(A_i)\ge v:= \displaystyle\frac{\mu(X)}{8c^2n}$, where $c$ is a positive constant depending on the covering property;
\item the family $\{2A_i\}_{i=1}^n$ is disjoint; and
\item the outer radius of each member in $\mathcal{A}$ is less than $1$.
\end{enumerate}
\end{lemma}

Lemma \ref{Asma} is a refinement of a result due to Grigor'Yan-Netrusov-Yau in \cite{Grigor'yan}. Roughly speaking, it establishes that if the measure $\mu$ has no atoms, then we can decompose the space $X$ into a finite family of disjoint annuli containing ``enough measure". 

\begin{remark}\label{cutlocus}
It is important to point out that by taking $n$ large enough, we can assume that the outer radius of the annuli is smaller than any given $\delta>0$. In particular, for an annulus center at $p$, we can assume that the outer boundary stays away from the cut locus of $p$.
\end{remark}
 
\begin{proof}[Proof of Theorem \ref{Main3}]
Let $d_{g_o}$ be the distance function induced by some background metric $g_o\in [g]$. As explained in Hassannezhad's work \cite{Hassannezhad}, the metric space $(M,d_{g_o})$ satisfies the $(2,N;1)$-covering property for some $N\in\mathbb{N}$ depending only on the dimension of the manifold. Let $g_w\in[g]$ be arbitrary. We apply Lemma \ref{Asma} to the metric-measure space $(M^4, d=d_{g_o}, \mu=dv_{g_w})$ with $n\ge \max\{k,k_M\}$ to get a family $\mathcal{A} = \{A_i\}_{i=1}^n$ of annuli satisfying properties $(i), (ii)$, and $(iii)$ above. 

Select a smooth function $\phi_i \in C^\infty(M^4)$ satisfying the following properties: (i) $\phi_i \equiv 1$ on $A_i = A(x_i,r_i,R_i)$ and $0\le\phi_i\le1$ everywhere; (ii) supp$(\phi_i)\subset 2A_i$; (iii) $|\nabla_{g_o}\phi_i| = O(\frac{1}{r})$ and $|\Delta_{g_o}\phi_i|=O(\frac{1}{r^2})$ on $A(x_i, \frac{r_i}{2},r_i)$; and (iv) $|\nabla_{g_o}\phi_i| = O(\frac{1}{R})$ and $|\Delta_{g_o}\phi_i|=O(\frac{1}{R^2})$ on $A(x_i, R_i,2R_i)$ (see Remark \ref{cutlocus}). Indeed, select smooth functions $f$ and $\tilde f$, with values between $0$ and $1$, such that the following holds:
\[
f\equiv 1\;\text{on }B_{g_o}(x_i,R_i);\; \text{supp}(f)\subseteq B_{g_o}(x_i,2R_i);\; |\nabla_{g_o}f|\le \frac{C}{R_i};\; |\Delta_{g_o}f|\le \frac{C}{R_i^2},
\]
and 
\[
\tilde f\equiv 1\;\text{on }B_{g_o}(x_i,r_i)^c;\; \text{supp}(\tilde f)\subseteq B_{g_o}\left(x_i,\frac{r_i}{2}\right)^c;\; |\nabla_{g_o}\tilde f|\le \frac{C}{r_i};\; |\Delta_{g_o}\tilde f|\le \frac{C}{r_i^2},
\]
where the constant $C$ denotes a generic constant that depends only on the background metric $g_o$. We can then check that $\phi_i = f\tilde f$ satisfies the desired properties.

For any of the functions $\phi_i$, it follows that the numerator in the Rayleigh quotient of $R_{g_w}(\phi_i)$ is estimated by a constant $C$ depending only of $g_o$:
\begin{equation} 
\int_M \phi_iP_{g_o}\phi_i\;dv_{g_o}\le C.
\end{equation}
The Rayleigh quotient with respect to the metric $g_w$ can now be estimated as follows:
\begin{equation}
R_{g_w}(\phi_i) = \frac{\displaystyle \int_M \phi_iP_{g_o}\phi_i\;dv_{g_o}}{\displaystyle \int_M \phi_i^2\;dv_{g_w}} \le C\cdot \frac{8c^2n}{\text{Vol}(M^4,g_w)}  = \frac{Cn}{\text{Vol}(M^4,g_w)}.
\end{equation}
Using the min-max characterization (\ref{min-maxP}) for $P_{g_w}$, we deduce
\begin{equation}
\lambda_k(P_{g_w})\text{Vol}(M^4,g_w)\le Cn.
\end{equation}
Since $g_w$ is arbitrary, the result follows. 
\end{proof}

\subsection{Hersch's type inequalities on $(\mathbb{S}^4, g_r)$}\label{Sphere}
This section is devoted to the proof of Theorems \ref{Main4} and \ref{third-eigen}. 
\subsubsection{The second eigenvalue}
Recall that by Corollary \ref{Homogeneous}, if $(M^4,g)$ is a homogeneous manifold, then $g$ would be conformally extremal for some eigenvalues. In particular, the round metric $g_r$ on $\mathbb{S}^4$ is conformally extremal for the normalized second eigenvalue functional. This is because $P_{g_r}$ is nonnegative with $\text{Ker}(P_{g_r}) = \{\text{constants}\}$, and therefore $\lambda_2(P_{g_w})>\lambda_1(P_{g_w}) = 0$ for every $g_w\in[g_r]$. Theorem \ref{Main4} says that $g_r$ is in fact maximal.

\begin{proof}[Proof of Theorem \ref{Main4}]

For any $\tilde g\in [g_r]$, we will use the following characterization for $\lambda_2(P_{\tilde g})$:
\begin{equation}\label{minmax-P2}
\lambda_2(P_{\tilde g}) = \inf_{\varphi \in C^\infty(\mathbb{S}^4);\; \int_{\mathbb{S}^4}\varphi\;dv_{\tilde g} = 0} R_{\tilde g}(\varphi).
\end{equation}
Let $\{x_1,\cdots,x_5\}$ be the standard coordinates on $\mathbb{R}^5$, and let $g_w = e^{2w}g_r$ be any metric in $[g_r]$ with volume equal to $\omega_4$. Using a balancing argument, we can find a conformal diffeomorphism $\varphi : \mathbb{S}^4\rightarrow \mathbb{S}^4$ such that for each $i=1,\cdots,5$,
\begin{equation}
\int_{\mathbb{S}^4} x_i\;dv_{\varphi^*g_w} = \int_{\mathbb{S}^4} e^{4w_\varphi}x_i\;dv_{g_r} =0, 
\end{equation}
where $\varphi^*g_w = e^{2w_\varphi}g_r$. From the variational characterization of $\lambda_2(P_{g_w})$ and the conformal invariance (\ref{energy}), we obtain
\begin{equation}\label{Main4-1}
\lambda_2(P_{\varphi^*g_w}) \int_{\mathbb{S}^4} x_i^2\;dv_{\varphi^*g_w} \le \int_{\mathbb{S}^4} x_iP_{\varphi^*g_w}x_i\;dv_{\varphi^*g_w} = \int_{\mathbb{S}^4} x_iP_{g_r}x_i\;dv_{g_r}.
\end{equation}
Summing up over $i$ gives
\begin{equation}\label{Main4-2}
\lambda_2(P_{\varphi^*g_w})\text{Vol}(\mathbb{S}^4,\varphi^*g_w) \le \sum_{i=1}^5 \int_{\mathbb{S}^4} x_iP_{g_r}x_i\;dv_{g_r}.
\end{equation}

Recall by (\ref{PSphere}) that the Paneitz operator with respect to the round metric $g_r$ reduces to
\begin{equation}
P_{g_r} = \Delta_{g_r}^2 - 2\Delta_{g_r}.
\end{equation}
Since $\Delta_{g_r}x_i = -4x_i$ for each $i=1,\cdots,5$, then
\begin{equation}
\begin{split}
0 &= \frac{1}{2}\Delta_{g_r}(1) =\frac{1}{2} \Delta_{g_r}\left(\sum_{i=1}^5 x_i^2\right) = \sum_{i=1}^5 \left(-4x_i^2 + |\nabla_{g_r} x_i|_{g_r}^2\right)\\ &= -4 + \sum_{i=1}^5|\nabla_{g_r} x_i|_{g_r}^2,
\end{split}
\end{equation}
Therefore, the right-hand side of (\ref{Main4-2}) equals
\[
\begin{split}
\sum_{i=1}^5 \int_{\mathbb{S}^4} x_iP_{g_r}x_i\;dv_{g_r} &= \sum_{i=1}^5\int_{\mathbb{S}^4} \left\{(\Delta_{g_r}x_i)^2 + 2|\nabla_{g_r}x_i|^2_{g_r}\right\}\;dv_{g_r} \\ &= \sum_{i=1}^5\int_{\mathbb{S}^4} \left\{16x_i^2 + 2|\nabla_{g_r}x_i|^2_{g_r} \right\}\;dv_{g_r} = \int_{\mathbb{S}^4} 24\;dv_{g_r} = 24\omega_4.
\end{split}
\]
Since the metrics $g_w$ and $\varphi^*g_w$ are isometric, we conclude from (\ref{Main4-2}) that
\begin{equation}
\lambda_2(P_{g_w}) \le 24.
\end{equation}

It remains to discuss the equality case. If $\lambda_2(P_{\varphi^*g_w})=24$, then from (\ref{Main4-1}) and (\ref{Main4-2}) we deduce that the coordinate functions $x_1,\cdots,x_5$ are also eigenfunctions for $P_{\varphi^*g_w}$ with a corresponding eigenvalue equal to $24$. Using the conformal properties (\ref{ConfTransP}) of $P_{\varphi^*g_w}$, we observe
\begin{equation}
e^{-4w_\varphi}(24x^i)=e^{-4w_\varphi}P_{g_r}x_i = P_{\varphi^*g_w}x_i = 24x_i.
\end{equation} 
Since this holds for every $i=1,\cdots,5$, the function $e^{4w}$ is everywhere equal to $1$. Therefore, $\varphi^*g_w = e^{2w_\varphi} g_r = g_r$, and hence both metrics are isometric. 
\end{proof}

\subsubsection{The third eigenvalue}

Note that the multiplicity of $\lambda_2(P_{g_r})$ is $5$. As a consequence of Corollary \ref{Consecutive}, we know that the round metric $g_r$ on $\mathbb{S}^4$ cannot maximize more than $4$ consecutive conformal eigenvalues. We state this observation as a proposition.

\begin{proposition}\label{Main5}
Let $g_r$ be the standard round metric on $\mathbb{S}^4$. Then $g_r$ cannot maximize at least one of the conformal eigenvalues in the following list: 
$\Lambda^P_3(\mathbb{S}^4,[g_r])$, $\Lambda^P_4(\mathbb{S}^4,[g_r])$, $\Lambda^P_5(\mathbb{S}^4,[g_r])$, or $\Lambda^P_6(\mathbb{S}^4,[g_r])$.
\end{proposition}
In fact, we expect that the value for any higher order conformal eigenvalue (i.e. $k\ge 3$) cannot be attained by a smooth metric in $[g_r]$ (cf. Theorem 1.2 in \cite{Karpukhin3}). Theorem \ref{third-eigen} is the first step in proving this conjecture. 

\begin{proof}[Proof of Theorem \ref{third-eigen}]
The proof uses a construction of trial functions, i.e. functions orthogonal to the constants and the coordinate functions, as done originally in the works of Nadirashvili (\cite{Nadirashvili4}) and Petrides (\cite{Petrides3}). We will follow mostly the notation used in \cite{Kim}.

Let us start by setting up some notation. Consider the following family $T_\xi:\mathbb{B}^5\rightarrow \mathbb{B}^5$ of Mobi\"us transformations parametrized by $\xi \in \mathbb{B}^5\subseteq \mathbb{R}^5$, where $\mathbb{B}^5$ denotes the closed unit ball in $\mathbb{R}^5$:
\begin{equation}
T_\xi(y) = \frac{(1+2\xi\cdot y+ |y|^2)\xi + (1-|\xi|^2)y}{1+2\xi\cdot y+|\xi|^2|y|^2}.
\end{equation}
These are conformal transformations when restricted to $\mathbb{S}^4$. On the other hand, for any $(p,t)\in \mathbb{S}^4\times [0,1)$, consider the family of spherical caps defined by 
\begin{equation}
H_p = H_{p,0} := \{y\in \mathbb{S}^4: y\cdot p \le 0\},
\end{equation}
and
\begin{equation}
H_{p,t}:= T_{pt}(H_{p,t}).
\end{equation}
Let $R_p(y) = y - 2(y\cdot p)p$, $y\in \mathbb{S}^4$ denote the reflection across the hyperplane through the origin and orthogonal to $p\in \mathbb{S}^4$, and denote by $R_{H_{p,t}} = T_{pt}\circ R_p\circ T_{-pt}: \mathbb{S}^4\rightarrow \mathbb{S}^4$ the conformal reflection across the boundary of $H_{p,t}$. Finally, the folding $F_{H_{p,t}}$w.r.t. $H_{p,t}$ is defined by 
\begin{equation}
F_{H_{p,t}}(y) =
\begin{cases}
y & y \in H_{p,t}\\
R_{H_{p,t}}(y) & y\in \mathbb{S}^4\setminus H_{p,t}.
\end{cases}
\end{equation}
The result follows from the following lemma:
\begin{lemma}[Proposition 8 in \cite{Kim}]
Let $f$ be an eigenfunction associated with $\lambda_2(P_{g_w})$. Then there exists a point $(p,t)\in \mathbb{S}^4\times [0,1]$ and $c_{H_{p,t}}\in \mathbb{B}^5$ such that 
\begin{equation}
\int_{\mathbb{S}^4} (x_i\circ T_{-c_{H_{p,t}}}\circ F_{H_{p,t}})(y)f(y)\;dv_{g_w} = \int_{\mathbb{S}^4} (x_i\circ T_{-c_{H_{p,t}}}\circ F_{H_{p,t}})(y)\;dv_{g_w} = 0.
\end{equation}
for each $i=1,\cdots, 5$.
\end{lemma}

Using the variational characterization of $\lambda_3(P_{g_w})$, we then have
\[
\begin{split}
\lambda_3(P_{g_w})\int_{\mathbb{S}^4} (x_i\circ &T_{-c_{H_{p,t}}}\circ F_{H_{p,t}})^2\;dv_{g_w}\\ &\le \int_{\mathbb{S}^4} (x_i\circ T_{-c_{H_{p,t}}}\circ F_{H_{p,t}})P_{g_r}(x_i\circ T_{-c_{H_{p,t}}}\circ F_{H_{p,t}})\;dv_{g_r}.
\end{split}
\]
Summing over $i=1,\cdots, 5$ gives
\[
\begin{split}
\lambda_3(P_{g_w})\cdot \omega_4 \le& \sum_{i=1}^5\int_{\mathbb{S}^4}(x_i\circ T_{-c_{H_{p,t}}}\circ F_{H_{p,t}})P_{g_r}(x_i\circ T_{-c_{H_{p,t}}}\circ F_{H_{p,t}})\;dv_{g_r} \\ = &\sum_{i=1}^5 \int_{H_{p,t}} (x_i\circ T_{-c_{H_{p,t}}}\circ F_{H_{p,t}})P_{g_r}(x_i\circ T_{-c_{H_{p,t}}}\circ F_{H_{p,t}})\;dv_{g_r} \\&+\sum_{i=1}^5\int_{\mathbb{S}^4\setminus H_{p,t}} (x_i\circ T_{-c_{H_{p,t}}}\circ F_{H_{p,t}})P_{g_r}(x_i\circ T_{-c_{H_{p,t}}}\circ F_{H_{p,t}})\;dv_{g_r} \\ \le& 2\sum_{i=1}^5\int_{\mathbb{S}^4} x_iP_{g_r}x_i\;dv_{g_r} = 2\cdot \lambda_2(P_{g_r})\cdot \omega_4 = 2\cdot 24\cdot \omega_4.
\end{split}
\]
This finishes the proof. 
\end{proof}

\section{Extremal metrics on $\mathcal{R}(M^4)$}\label{ExtremalMetrics}

We turn our attention to extremal metrics for the normalized eigenvalue functional, but over the space of all Riemannian metrics. We will refer to critical points of the normalized eigenvalue functional 
\begin{equation}
g\in \mathcal{R}(M^4)\mapsto \lambda_k(P_g)\text{Vol}(M^4,g)
\end{equation}
as \textit{extremal metrics} (see Definition \ref{Extremal}).

First, we briefly summarized what is known for laplacian eigenvalues on closed surfaces. 
\begin{enumerate}
\item By results in \cite{Yang, Korevaar, Karpukhin}, we know that 
\[
\sup_{g\in \mathcal{R}(M^2)}\lambda_k(-\Delta_g)\text{Area}(M^2,g)<\infty
\]
for each $k\in \mathbb{N}$, where $\mathcal{R}(M^2)$ is the space of all Riemannian metrics. 

\item If $g_e\in \mathcal{R}(M^2)$ is extremal for the normalized eigenvalue functional $g$ $\in \mathcal{R}(M^2)\mapsto \lambda_k(-\Delta_g)\text{Area}(M^2,g)$, then there exists a collection $\{\phi_1,\cdots, \phi_p\}$ $\subseteq E_k(-\Delta_g)$ such that $\sum_{i=1}^p\phi^2$ is constant, and $\sum_{i=1}^p d\phi_i\otimes d\phi_i = g$ (\cite{Soufi2}) (see Theorem 3.1 in \cite{Soufi2}). This last condition means that $\Phi$ is an isometric minimal immersion into $\mathbb{S}^{p-1}$ (\cite{Takahashi}).

\item If either $\lambda_k(-\Delta)>\lambda_{k-1}(-\Delta_g)$ or $\lambda_k(-\Delta_g)<\lambda_{k+1}(-\Delta_g)$, then the necessary conditions in (2) are also sufficient for the existence of extremal metrics (see Theorem 3.1 in \cite{Soufi2}). 
\end{enumerate}

Theorem \ref{E-Main}, stated in the introduction, provides the analog for the Paneitz operator on closed four-manifolds of items (2) and (3) above. The proof of Theorem \ref{E-Main} is contained in Proposition \ref{E-nec} and Proposition \ref{E-suf}. The corresponding result to (1) for the Paneitz operator is still not available, and it would be interesting to investigate if such topological bounds can be found in this case. The author plans to investigate this and related questions in future works. 

Finally, we would like to point out that some of the results stated for conformally extremal metrics can be deduced as a consequence of the propositions in this section; see Remark \ref{Gen-to-Conf} for instance. We have decided, however, to keep both sections separate as some readers might be interested only in the conformal case, and the results in this section require familiarity with tensorial calculations. 

\subsection{Proof of Theorem \ref{E-Main}}

Let $g(t) \in \mathcal{R}(M^4)$ be any one-parameter family of metrics which is analytic in a neighborhood of $t=0$, and set $g=g(0)$. 

\begin{proposition}\label{E-OSD}
For any such analytic perturbation $g(t)\in \mathcal{R}(M^4)$ of $g$, the one-sided derivatives of $t\mapsto \lambda_k(P_{g(t)})$ at $t=0$ always exist. Moreover, we have the same formulas as in Proposition \ref{OSD}; that is, if $\lambda_k(P_g)>\lambda_{k-1}(P_g)$, then (\ref{OSD+1}) and (\ref{OSD-1}) hold, while if $\lambda_k(P_g)<\lambda_{k+1}(P_g)$, then (\ref{OSD+2}) and (\ref{OSD-2}) hold. 
\end{proposition}

\begin{proof}
Notice that the first part (existence of one-sided derivatives and formulas (\ref{OSD+1}), (\ref{OSD-1}), (\ref{OSD+2}), and (\ref{OSD-2})) in the proof of Proposition \ref{OSD} does not use in any manner that the perturbation is conformal. Therefore, the arguments also work in this case. 
\end{proof}

Proposition \ref{E-OSD} allow us to define critical points for the normalized eigenvalue functional 
\begin{equation}
g\in \mathcal{R}(M^4) \mapsto \lambda_k(P_g)\text{Vol}(M^4,g).
\end{equation}

\begin{definition}\label{Extremal}
A metric $g\in \mathcal{R}(M^4)$ is said to be \textit{extremal for $\lambda_k$} if and only if for any volume-preserving perturbation $g(t)\in \mathcal{R}(M^4)$ of $g$ which is analytic in a neighborhood of $t=0$, we have 
\begin{equation}
\frac{d}{dt}\lambda_k(P_{g(t)})|_{t=0^+}\cdot \frac{d}{dt}\lambda_k(P_{g(t)})|_{t=0^-}\le 0.
\end{equation}
\end{definition}

Suppose $g(t)\in \mathcal{R}(M^4)$ is an analytic perturbation of $g$, set $h=\frac{d}{dt}g(t)|_{t=0}$, and let $m$ be the multiplicity of $\lambda_k(P_g)$. Following the notation of Proposition \ref{RKC}, we claim that $\{\Lambda_1'(0),\cdots, \Lambda_m'(0)\}$ are the eigenvalues of the operator $\Pi_k \circ P_g'$, where $\Pi_k$ is the projection onto $E_k(P_g)$ and $P_g'=\frac{d}{dt}P_{g(t)}|_{t=0}$ is the linearization of $P_g$ with respect to $h$. 

\begin{proposition}\label{E-quadraticform}
The operator $\phi\in E_k(P_g)\mapsto Q_h(\phi) := \displaystyle\int_M \phi P_g'\phi\;dv_g$ defines a quadratic form on $E_k(P_g)$ given by
\begin{equation}
Q_h(\phi) = -\int_M \left\langle h, \tau_g(\phi) + \frac{1}{2}\lambda_k(P_g)\phi^2g\right\rangle \;dv_g,
\end{equation}
where $\tau_g(\phi)$ is the trace-free covariant two-tensor
\begin{equation}\label{Two-Tensor}
\begin{split}
\tau_g(\phi) :=& -2d(\Delta_g \phi)\otimes d\phi + (\Delta_g\phi)^2g+\langle \nabla_g\Delta_g\phi, \nabla_g\phi\rangle g\\&-\frac{2}{3} \nabla (d|\nabla_g\phi|^2) + \frac{2}{3}\Delta_g|\nabla_g\phi|^2g+\frac{2}{3}|\nabla_g\phi|^2\text{Ric}_g + \frac{2}{3}R_gd\phi\otimes d\phi \\ &- \Delta(d\phi\otimes d\phi) +2\nabla(\delta(d\phi \otimes d\phi)) - \delta^2(d\phi\otimes d\phi) g \\ &- 2Rm(\nabla_g\phi, \cdot,\nabla_g\phi,\cdot)- 2\text{Ric}_g(\nabla_g\phi,\cdot)\otimes d\phi -\frac{1}{2}E_g(\phi,\phi)g,
\end{split}
\end{equation}
and
\begin{equation}
E_g(\phi,\varphi)=(\Delta_g\phi)(\Delta_g\varphi)+ \frac{2}{3}R_g\langle \nabla_g \phi,\nabla_g\varphi\rangle - 2\text{Ric}_g(\nabla_g \phi, \nabla_g\varphi).
\end{equation}
Moreover, $\{\Lambda_1'(0),\cdots, \Lambda_m'(0)\}$ are the eigenvalues of $\Pi_k\circ P_g'$:
\begin{equation}\label{E-quadraticform1}
Q_h(\phi)=\int_M\phi_iP_g'\phi_i\;dv_g = \Lambda_i'(0)
\end{equation}
for each $i\in \{1,\cdots, m\}$. 
\end{proposition}

\begin{proof} We divide the proof into three steps: 

\noindent{\bf Step 1:} Here we prove that the operator
\begin{equation}
(\phi, \varphi)\in E_k(P_g)\mapsto \int_M\phi P_g'\varphi \;dv_g
\end{equation}
defines a quadratic symmetric form on $E_k(P_g)$; that is, the operator $\Pi_k\circ P_g'$ is symmetric on $L^2(M^4,g)$. 

We first need the following lemma about the linearization of different geometrically defined quantities. The reader can consult the notes by J. Viaclovsky in \cite{Viaclovsky} where most of the details can be found. 
\begin{lemma}\label{linearization}
The linearization of the Ricci curvature is given by 
\begin{equation}
\begin{split}
(\text{Ric}_g')_{ab}:= \frac{d}{dt}(\text{Ric}_{g(t)})_{ab}|_{t=0} &= \frac{1}{2}[-(\Delta h)_{ab} + \nabla_a(\delta h)_b + \nabla_b(\delta h)_a\\ & -\nabla_a\nabla_b(tr_gh)] -g^{kp}g^{lq}R_{akbl}h_{pq}\\ &+ \frac{1}{2} g^{sp}(\text{Ric}_g)_{ap} h_{bs} + \frac{1}{2} g^{sp}(\text{Ric}_g)_{bp} h_{as},
\end{split}
\end{equation}
where $\Delta$ and $\delta$ denotes the rough laplacian and divergence operator\footnote{For instance, $(\delta h)_a = g^{kl}\nabla_kh_{la}$. In particular, there is no minus sign in front.} acting on tensors, respectively; $R_{akbl}$ denotes the components of the Riemannian curvature tensor $Rm$, and $\nabla$ is the covariant derivative induced by $g$ on tensor bundles. The linearization of the scalar curvature is given by 
\begin{equation}
R_g':=\frac{d}{dt}R_{g(t)}|_{t=0} = -\Delta_g(tr_gh) +\delta^2 h -\langle \text{Ric}_g,h\rangle,
\end{equation}
where $\langle \cdot, \cdot\rangle$ is the inner product induced by $g$ on the space $S^2(T^*M)$ of covariant two-tensors. The linearization of the Laplace-Beltrami operator acting on a function $\phi\in C^\infty(M^4)$ is
\begin{equation}
\Delta_g'\phi:=\frac{d}{dt} \Delta_{g(t)}|_{t=0} \phi= -\langle h,\nabla(d\phi)\rangle - \langle \delta h, d\phi\rangle + \frac{1}{2}\langle d(tr_gh), d\phi\rangle.
\end{equation}
Finally, the linearization of the volume element is given by 
\begin{equation}
\frac{d}{dt} dv_{g(t)}|_{t=0} = \frac{1}{2} (tr_gh)\;dv_g.
\end{equation}
\end{lemma}

Let $\phi, \varphi \in E_k(P_g)$ be arbitrary. To avoid computing the linearization of $P_g$ directly, we proceed as follows:
\[
\begin{split}
&\int_M \phi P_g'\varphi \; dv_g = \frac{d}{dt}\int_M \phi P_{g(t)}\varphi\; dv_{g(t)}|_{t=0} - \int_M \phi P_g \varphi \frac{d}{dt}dv_{g(t)}|_{t=0} \\ & =\int_M\frac{d}{dt}E_{g(t)}(\phi,\varphi)|_{t=0}\;dv_{g} + \int_ME_g(\phi,\varphi)\frac{1}{2}(tr_gh)\;dv_{g}  - \lambda_k(P_g)\int_M \phi\varphi\frac{1}{2}(tr_gh)\;dv_g,
\end{split}
\]
where Lemma \ref{linearization} has been used to compute the linearization of the volume element. Notice that we are done with the last two terms, both of which are symmetric in $\phi$ and $\varphi$.

Let us now focus on the term $\displaystyle\int_M\frac{d}{dt}E_{g(t)}(\phi,\varphi)|_{t=0}\;dv_{g}$. We start with
\[
\begin{split}
\int_M& \frac{d}{dt}(\Delta_{g(t)}\phi)(\Delta_{g(t)}\varphi)|_{t=0}\;dv_g = \int_M(\Delta_g'\phi) \Delta_g\varphi\;dv_g + \int_M(\Delta_g\phi) \Delta_g'\varphi\;dv_g\\ =& \int_M \left\{-\langle h, (\Delta_g\varphi)\nabla(d\phi)\rangle - \langle \delta h, (\Delta_g \varphi)d\phi\rangle +\frac{1}{2}\langle d(tr_gh),(\Delta_g\varphi)d\phi\rangle \right\}\;dv_g\\ &+\int_M \left\{-\langle h, (\Delta_g\phi)\nabla(d\varphi)\rangle - \langle \delta h, (\Delta_g \phi)d\varphi\rangle +\frac{1}{2}\langle d(tr_gh),(\Delta_g\phi)d\varphi\rangle \right\}\;dv_g \\=& \int_M \left\{-\langle h, (\Delta_g\varphi)\nabla (d\phi)\rangle + \langle h, \nabla((\Delta_g\varphi)d\phi)\rangle -\frac{1}{2}(tr_gh)\text{div}_g((\Delta_g\varphi) \nabla_g\phi)\right\}\;dv_g \\ &+ \int_M \left\{-\langle h, (\Delta_g\phi)\nabla (d\varphi)\rangle + \langle h, \nabla((\Delta_g\phi)d\varphi)\rangle -\frac{1}{2}(tr_gh)\text{div}_g((\Delta_g\phi) \nabla_g\varphi)\right\}\;dv_g \\=& \int_M \langle h, d(\Delta_g \varphi)\otimes d\phi + d(\Delta_g\phi)\otimes d\varphi - (\Delta_g \varphi)(\Delta_g\phi)g - \frac{1}{2}\langle \nabla_g\Delta_g\varphi, \nabla_g\phi\rangle g \rangle\;dv_g \\&-\frac{1}{2}\int_M\langle h, \langle \nabla_g\Delta_g\phi, \nabla_g\phi\rangle g\rangle \;dv_g,
\end{split}
\]
where we have used the identity $\nabla((\Delta_g\phi)d\varphi) = d(\Delta_g\phi)\otimes d\varphi + (\Delta_g\phi)\nabla(d\varphi)$ after applying the divergence theorem. For the term involving the scalar curvature, we have
\[
\begin{split}
\frac{2}{3}&\int_M\frac{d}{dt}R_{g(t)}\langle \nabla_{g(t)}\phi, \nabla_{g(t)}\varphi\rangle|_{t=0}\;dv_g \\ =&\frac{2}{3}\int_M \left\{-\Delta_g(tr_gh)+\delta^2h-\langle \text{Ric}_g, h\rangle \right\}  \langle\nabla_g\phi, \nabla_g\varphi\rangle\;dv_g\\ &-\frac{2}{3}\int_M R_g\langle h, d\phi\otimes d\varphi\rangle \;dv_g \\  =&\frac{2}{3}\int_M \langle h, \nabla(d \langle \nabla_g\phi,\nabla_g\varphi \rangle) - \Delta_g\langle \nabla_g\phi,\nabla_g\varphi \rangle g - \langle \nabla_g\phi,\nabla_g\varphi \rangle \text{Ric}_g\rangle\;dv_g\\ &-\frac{2}{3} \int_M\langle h, R_gd\phi\otimes d\varphi\rangle \;dv_g.
\end{split}
\]
We now deal with the term involving the Ricci curvature. First, notice that
\[
\text{Ric}_{g(t)}(\nabla_{g(t)}\phi, \nabla_{g(t)}\varphi) = (\text{Ric}_{g(t)})_{ab}g(t)^{ac}g(t)^{bd}\phi_c \varphi_d.
\]
Using $(g')^{ac} = -g^{ae}g^{cf}h_{ef}$, we then deduce
\[
\begin{split}
\frac{d}{dt}& \text{Ric}_{g(t)}(\nabla_{g(t)}\phi, \nabla_{g(t)}\varphi)|_{t=0} = (\text{Ric}_g')_{ab}g^{ac}g^{bd}\phi_c \varphi_d + (\text{Ric}_g)_{ab}(g')^{ac}g^{bd}\phi_c\varphi_d\\ &\hspace{1.65in}+ (\text{Ric}_g)_{ab}g^{ac}(g')^{bd}\phi_c\varphi_d\\=& \frac{1}{2}[-(\Delta h)_{ab} + \nabla_a(\delta h)_b + \nabla_b(\delta h)_a - \nabla_a\nabla_b(tr_gh)] g^{ac}g^{bd}\phi_c \varphi_d\\ &-g^{kp}g^{lq}R_{akbl}h_{pq}g^{ac}g^{bd}\phi_c \varphi_d+ \frac{1}{2} g^{sp}(\text{Ric}_g)_{ap} h_{bs}g^{ac}g^{bd}\phi_c \varphi_d\\ &+ \frac{1}{2} g^{sp}(\text{Ric}_g)_{bp} h_{as}g^{ac}g^{bd}\phi_c \varphi_d + (\text{Ric}_g)_{ab}(g')^{ac}g^{bd}\phi_c\varphi_d + (\text{Ric}_g)_{ab}g^{ac}(g')^{bd}\phi_c\varphi_d \\=& -\frac{1}{2}\langle \Delta h, d\phi\otimes d\varphi\rangle +\langle \nabla(\delta h), d\phi\otimes d\varphi\rangle -\frac{1}{2}\langle \nabla^2(tr_gh), d\phi \otimes d\varphi\rangle\\ &- \langle h, Rm(\nabla_g\phi, \cdot, \nabla_g\varphi, \cdot)\rangle +\frac{1}{2} \langle h, \text{Ric}_g(\nabla_g\phi,\cdot)\otimes d\varphi\rangle +\frac{1}{2} \langle h, \text{Ric}_g(\nabla_g\varphi,\cdot)\otimes d\phi\rangle \\ &-(\text{Ric}_g)_{ab}g^{ae}g^{cf}h_{ef}g^{bd}\phi_c\varphi_d - (\text{Ric}_g)_{ab}g^{ac}g^{be}g^{df}h_{ef}\phi_c\varphi_d \\ =& -\frac{1}{2}\langle \Delta h, d\phi\otimes d\varphi\rangle +\langle \nabla(\delta h), d\phi\otimes d\varphi\rangle -\frac{1}{2}\langle \nabla^2(tr_gh), d\phi \otimes d\varphi\rangle\\ &- \langle h, Rm(\nabla_g\phi, \cdot, \nabla_g\varphi, \cdot)\rangle - \frac{1}{2}\langle h, \text{Ric}_g(\nabla_g\phi,\cdot)\otimes d\varphi\rangle - \frac{1}{2}\langle h, \text{Ric}_g(\nabla_g\varphi,\cdot)\otimes d\phi\rangle.
\end{split}
\]
Therefore, 
\[
\begin{split}
-&2\int_M\frac{d}{dt} \text{Ric}_{g(t)}(\nabla_{g(t)}\phi, \nabla_{g(t)}\varphi)|_{t=0}\;dv_g \\ =& \int_M \langle h, \Delta(d\phi\otimes d\varphi) - 2\nabla(\delta(d\phi\otimes d\varphi)) + \delta^2(d\phi \otimes d\varphi) g \rangle \; dv_g \\&+\int_M \langle h, 2Rm(\nabla_g\phi,\cdot, \nabla_g \varphi, \cdot) + \text{Ric}_g(\nabla_g\phi,\cdot)\otimes d\varphi + \text{Ric}_g(\nabla_g\varphi,\cdot)\otimes d\phi\rangle\;dv_g.
\end{split}
\]
Putting everything together, we have that $\displaystyle\int_M \phi P_g' \varphi \; dv_g$ equals 
\[
\begin{split}
&\int_M \langle h, d(\Delta_g \varphi)\otimes d\phi + d(\Delta_g\phi)\otimes d\varphi - (\Delta_g \varphi)(\Delta_g\phi)g - \frac{1}{2}\langle \nabla_g\Delta_g\varphi, \nabla_g\phi\rangle g \rangle\;dv_g \\&-\frac{1}{2}\int_M\langle h, \langle \nabla_g\Delta_g\phi, \nabla_g\phi\rangle g\rangle \;dv_g \\  &+\frac{2}{3}\int_M \langle h, \nabla(d \langle \nabla_g\phi,\nabla_g\varphi \rangle) - \Delta_g\langle \nabla_g\phi,\nabla_g\varphi \rangle g - \langle \nabla_g\phi,\nabla_g\varphi \rangle \text{Ric}_g\rangle\;dv_g\\ &-\frac{2}{3} \int_M\langle h, R_gd\phi\otimes d\varphi\rangle \;dv_g\\&+ \int_M \langle h, \Delta(d\phi\otimes d\varphi) - 2\nabla(\delta(d\phi\otimes d\varphi)) + \delta^2(d\phi \otimes d\varphi) g \rangle \; dv_g \\&+\int_M \langle h, 2Rm(\nabla_g\phi,\cdot, \nabla_g \varphi, \cdot) + \text{Ric}_g(\nabla_g\phi,\cdot)\otimes d\varphi + \text{Ric}_g(\nabla_g\varphi,\cdot)\otimes d\phi\rangle\;dv_g \\&+\frac{1}{2}\int_M \langle h, E_g(\phi,\varphi)g - \lambda_k(P_g)\phi\varphi g\rangle \;dv_g.
\end{split}
\]
This shows that $\displaystyle\int_M\phi P_g'\varphi\;dv_g$ defines a quadratic symmetric form on $E_k(P_g)$. The associated quadratic form $Q_h(\phi)$ is then given by 
\[
\begin{split}
Q_h(\phi) =& \int_M \langle h, 2d(\Delta_g \phi)\otimes d\phi - (\Delta_g\phi)^2g-\langle \nabla_g\Delta_g\phi, \nabla_g\phi\rangle g\rangle \;dv_g \\ &+ \frac{2}{3}\int_M \langle h, \nabla (d|\nabla_g\phi|^2) - \Delta_g|\nabla_g\phi|^2g-|\nabla_g\phi|^2\text{Ric}_g - R_gd\phi\otimes d\phi\rangle \;dv_g \\ &+ \int_M\langle h, \Delta(d\phi\otimes d\phi) -2\nabla(\delta(d\phi \otimes d\phi)) + \delta^2(d\phi\otimes d\phi) g\rangle\;dv_g \\ &+\int_M \langle h, 2Rm(\nabla_g\phi, \cdot,\nabla_g\phi,\cdot)+ 2\text{Ric}_g(\nabla_g\phi,\cdot)\otimes d\phi\rangle\;dv_g \\ &+\frac{1}{2}\int_M\langle h, E_g(\phi,\phi)g - \lambda_k(P_g)\phi^2 g\rangle dv_g,
\end{split}
\]
as desired. 

\noindent{\bf Step 2:} We prove that the eigenvalues of $\Pi\circ P_g'$ are $\{\Lambda_1'(0),\cdots,\Lambda_m'(0)\}$. 

The proof is similar to that of Proposition \ref{OSD}, but here the perturbation is not necessarily conformal. Starting from 
\begin{equation}
P_g(t)\phi_i(t)= \Lambda_i(t)\phi_i(t),
\end{equation}
we differentiate to obtain
\begin{equation}
P_g'\phi_i + P_g\phi_i' = \Lambda_i'(0)\phi_i + \lambda_k(P_g)\phi_i',
\end{equation}
where $\phi_i = \phi_i(0)$ and $\phi_i' = \phi_i'(0)$. Multiplying across by $\phi_j$ and subtracting $\phi_jP_g\phi_i' = \lambda_k(P_g)\phi_j\phi_i'$, we deduce
\begin{equation}
\begin{split}
\int_M(\phi_jP_g'\phi_i + \phi_jP_g\phi_i' - \phi_jP_g\phi_i')\; dv_g &= \Lambda_i'(0)\delta_{ij} + \lambda_k(P_g)\int_M \phi_j\phi_i'\;dv_g \\ &\hspace{.65in}- \lambda_k(P_g)\int_M \phi_j\phi_i'\;dv_g.
\end{split}
\end{equation}
Using that $P_g$ is a self-adjoint operator, we conclude 
\begin{equation}
\int_M\phi_jP_g'\phi_i\;dv_g = \Lambda_i'(0)\delta_{ij}.
\end{equation}
Since $\{\phi_i\}_{i=1}^m$ is a basis for $E_k(P_g)$, the proof follows.  

\noindent{\bf Step 3:} We prove that $tr_g(\tau_g(\phi)) = \langle g, \tau_g(\phi)\rangle = 0$. In order to complete this step, we need the following lemma:

\begin{lemma}\label{DoubleDiv}
We have the following identity for the double-divergence acting on the symmetric two tensor $d\phi\otimes d\phi$:
\begin{equation}\label{DoubleDiv1}
\begin{split}
\delta^2(d\phi\otimes d\phi) & = 2\langle \nabla_g \Delta_g\phi, \nabla_g \phi\rangle + \text{Ric}_g(\nabla_g\phi, \nabla_g\phi) + |\nabla_g^2\phi|^2+ (\Delta_g\phi)^2\\& = \Delta_g|\nabla_g\phi|^2 -|\nabla^2_g\phi|^2-\text{Ric}_g(\nabla_g\phi,\nabla_g\phi) + (\Delta_g\phi)^2. 
\end{split}
\end{equation}
\end{lemma}
\noindent See Appendix \ref{Appendix} for the proof of Lemma \ref{DoubleDiv}.

Using Lemma \ref{DoubleDiv} and Bochner's formula, we compute
\[
\begin{split}
tr_g(\tau_g(\phi)) =& -2\langle \nabla_g\Delta_g\phi, \nabla_g\phi\rangle +4(\Delta_g\phi)^2 +4\langle \nabla_g\Delta_g \phi, \nabla_g \phi\rangle - \frac{2}{3}\Delta_g|\nabla_g\phi|^2 \\ &+\frac{8}{3}\Delta_g|\nabla_g\phi|^2 + \frac{2}{3}|\nabla_g\phi|^2R_g +\frac{2}{3}R_g|\nabla_g\phi|^2 - \Delta_g|\nabla_g\phi|^2 \\ &+2\delta^2(d\phi\otimes d\phi) - 4\delta^2(d\phi\otimes d\phi) - 2\text{Ric}_g(\nabla_g\phi,\nabla_g\phi) \\&- 2\text{Ric}_g(\nabla_g\phi, \nabla_g \phi)-2E_g(\phi,\phi)\\ =& -2\langle \nabla_g\Delta_g\phi,\nabla_g\phi\rangle +2(\Delta_g\phi)^2+\frac{4}{3}R_g|\nabla_g\phi|^2 -6\text{Ric}_g(\nabla_g\phi,\nabla_g\phi)\\&- 2E_g(\phi,\phi) - 2|\nabla_g^2\phi|^2 + \Delta_g|\nabla_g\phi|^2\\ =& -2\langle\nabla_g\Delta_g,\nabla_g\phi\rangle - 2\text{Ric}_g(\nabla_g\phi, \nabla_g\phi) - 2|\nabla_g^2\phi|^2 + \Delta_g|\nabla_g\phi|^2\\=&0.
\end{split}
\]
\end{proof}

\begin{remark}\label{Gen-to-Conf}
Equation (\ref{AnalyticEigen}) in Proposition \ref{OSD} can be derived as a consequence of Proposition \ref{E-OSD}. Indeed, if $g(t)=e^{2w_t}g$ is an analytic conformal perturbation of $g$, then $h=2\frac{d}{dt}w_t|_{t=0}g = 2 w'\cdot g$, and so
\begin{equation}
\begin{split}
\Lambda_i'(0) &= \int_M \phi_iP_g'\phi_i\;dv_g = -\int_M \left\langle h, \tau_g(\phi_i) + \frac{1}{2}\lambda_k(P_g)\phi_i^2g\right\rangle\;dv_g \\&= -2\int_Mw'(tr_g(\tau_g(\phi_i)) + 2\lambda_k(P_g)\phi_i^2)\;dv_g\\ &= -4\lambda_k(P_g)\int_Mw'\phi_i^2\;dv_g.
\end{split}
\end{equation}
This concludes the proof of the claim since $\alpha = 4w'$.
\end{remark}

\begin{proposition}\label{E-nec}
Let $g\in \mathcal{R}(M^4)$ be a Riemannian metric for which $\lambda_k(P_g)\not = 0$ and for which either $\lambda_k(P_g)>\lambda_{k-1}(P_g)$ or $\lambda_k(P_g)<\lambda_{k+1}(P_g)$. If the metric $g$ is extremal for the normalized $k$-th eigenvalue functional, then there exists  a collection $\{\phi_1,\cdots, \phi_p\}\subseteq E_k(P_g)$ such that 
\begin{equation}\label{E-nec1}
\sum_{i=1}^p\tau_g(\phi_i) =0
\end{equation}
and 
\begin{equation}\label{E-nec2}
\sum_{i=1}^p\phi^2=\frac{2}{|\lambda_k(P_g)|}.
\end{equation}
\end{proposition}

\begin{proof}
The proof of (\ref{E-nec1}) follows the arguments of \cite{Soufi,Soufi2,Fraser}, and it is similar to the proof of Theorem \ref{Main1} in Section \ref{ProofMain1}. We include the main steps. Consider the convex hull $K$ in $L^2(S^2(T^*M))$ of the set 
\begin{equation}
\left\{\tau_g(\phi)+\frac{1}{2}\lambda_k(P_g)\phi^2g: \phi\in E_k(P_g)\text{ and }\|\phi\|_{L^2(M,g)}=1\right\}.
\end{equation}
We claim that $\text{sign}(\lambda_k(P_g))g\in K$. 

We proceed by contradiction. If $\text{sign}(\lambda_k(P_g))g\not\in K$, then we use Hahn-Banach separation theorem to obtain an $h\in L^2(S^2(T^*M))$, which we can assume is smooth, such that 
\begin{equation}\label{E-nec3}
\text{sign}(\lambda_k(P_g))\int_M\langle h,g\rangle\;dv_g>0
\end{equation}
and 
\begin{equation}\label{E-nec4}
\int_M \left\langle h, \tau_g(\phi)+\frac{1}{2}\lambda_k(P_g)\phi^2g\right\rangle \; dv_g<0 
\end{equation} 
for all nonzero $\phi\in E_k(P_g)$. Set 
\begin{equation}
\bar h = h - (4\text{Vol}(M^4,g))^{-1} \int_M \langle h,g\rangle\;dv_g\cdot g.
\end{equation}
Then, similar to our choice in (\ref{choice-perturbation}), we can find a volume-preserving perturbation $g(t)$ of $g$ with $\bar h = \frac{d}{dt}g(t)|_{t=0}$. Recall that, by assumption, the metric $g$ is extremal, and thus $Q_{\bar h}$ must be indefinite on $E_k(P_g)$ in light of Proposition \ref{E-OSD} and Proposition \ref{E-quadraticform}. However, using (\ref{E-nec3}), (\ref{E-nec4}) and that $tr_g(\tau_g(\phi))=0$ for any $\phi\in E_k(P_g)$, we have
\begin{equation}
\begin{split}
Q_{\bar h}(\phi) =& -\int_M \left\langle \bar h, \tau_g(\phi) +\frac{1}{2}\lambda_k(P_g)\phi^2g\right\rangle\;dv_g \\=& -\int_M\left\langle h, \tau_g(\phi) +\frac{1}{2}\lambda_k(P_g)\phi^2g\right\rangle\;dv_g \\ &+ \frac{1}{2\text{Vol}(M^4,g)}|\lambda_k(P_g)|\int_M\phi^2\;dv_g\cdot \text{sign}(\lambda_k(P_g)) \int_M\langle h,g\rangle\;dv_g \\ &>0
\end{split}
\end{equation}
for all $\phi\in E_k(P_g)$. This means that $Q_{\bar h}$ is indefinite, which is a contradiction. 

Now, $\text{sign}(\lambda_k(P_g))g\in K$ means that there exists a collection $\{\phi_1\cdots,\phi_p\}\subseteq E_k(P_g)$ such that 
\begin{equation}
\sum_{i=1}^p\left(\tau_g(\phi_i) +\frac{1}{2}\lambda_k(P_g)\phi_i^2 g\right) = \text{sign}(\lambda_k(P_g))g.
\end{equation}
Taking the trace on both sides yields (\ref{E-nec2}), and so $\sum_{i=1}^p\tau_g(\phi_i)=0$.
\end{proof}

\begin{proposition}\label{E-suf}
Let $g\in \mathcal{R}(M^4)$ be a Riemannian metric for which $\lambda_k(P_g)\not = 0$ and for which either $\lambda_k(P_g)>\lambda_{k-1}(P_g)$ or $\lambda_k(P_g)<\lambda_{k+1}(P_g)$. If there exists a collection $\{\phi_1,\cdots, \phi_p\}\subseteq E_k(P_g)$ such that $\sum_{i=1}^p\phi_i^2$ is constant and $\sum_{i=1}^p\tau_g(\phi_i) = 0$,
then $g$ is extremal for the normalized $k$-th eigenvalue functional. 
\end{proposition}

\begin{proof}
Take any analytic perturbation $g(t)\in \mathcal{R}(M^4)$ of $g$ which preserves volume, and set $h=\frac{d}{dt}g(t)|_{t=0}\in S^2(T^*M)$. Since $g(t)$ is volume-preserving, we have
\begin{equation}
0=\frac{d}{dt}\text{Vol}(M^4,g(t))|_{t=0} = \frac{1}{2}\int_M tr_gh\;dv_g.
\end{equation}
Therefore, 
\begin{equation}
\begin{split}
\sum_{i=1}^p Q_h(\phi_i) &= -\sum_{i=1}^p\int_M\left\langle h,\tau_g(\phi_i)+\frac{1}{2}\lambda_k(P_g)\phi_i^2g\right\rangle\;dv_g\\ &= -\frac{1}{2}\lambda_k(P_g)\sum_{i=1}^p\phi_i^2\int_Mtr_gh\;dv_g = 0.
\end{split}
\end{equation}
Since $Q_h(\phi_i)=\Lambda_i'(0)$ by (\ref{E-quadraticform1}) in Proposition \ref{E-quadraticform}, we conclude that $Q_h$ is indefinite in $E_k(P_g)$ for each volume-preserving perturbation of $g$. By Proposition \ref{E-OSD}, specifically by formulas (\ref{OSD+1}), (\ref{OSD-1}), (\ref{OSD+2}), and (\ref{OSD-2}), this means precisely that $g$ is extremal for the normalized $k$-th eigenvalue functional. 
\end{proof}

\begin{remark}
In the case of closed surfaces for laplacian eigenvalues, the condition $\sum_{i=1}^pd\phi_i\otimes d\phi_i = g$ (see section \ref{E-MainResults}) alone implies that $\sum_{i=1}^p\phi_i^2$ is constant, and so we have a map into a sphere; see proof of Lemma 3.1 in \cite{Soufi2}. By further studying the condition $\sum_{i=1}^p\tau_g(\phi_i)=0$, we expect to be able to drop this assumption in the hypothesis of Proposition \ref{E-suf}.
\end{remark}

\section{Appendix: Proof of Lemma \ref{DoubleDiv}}\label{Appendix}

The double divergence of a symmetric two tensor $h\in S^2(T^*M)$ is defined as
\[
\delta^2h = \nabla^a\nabla^bh_{ab} = g^{ac}g^{bd}\nabla_c\nabla_d h_{ab}.
\]
We are interested in the particular case where $h=d\phi\otimes d\phi$. First, let us recall the following identity known as Ricci identity:
\[
\nabla_c\nabla_d(d\phi)_a - \nabla_d\nabla_c (d\phi)_a =R_{cdae}g^{me}(d\phi)_m
\]
Using Ricci identity, we compute as follows:
\[
\begin{split}
\delta^2(d\phi\otimes d\phi) =& g^{ac}g^{bd}\nabla_c\nabla_d(d\phi\otimes d\phi)_{ab}\\ =& g^{ac} g^{bd} \nabla_c(\nabla_d(d\phi)\otimes d\phi + d\phi \otimes \nabla_d(d\phi))_{ab} \\ =& g^{ac} g^{bd} (\nabla_c\nabla_d(d\phi)\otimes d\phi + \nabla_d(d\phi)\otimes \nabla_c(d\phi)+ \nabla_c (d\phi)\otimes\nabla_d(d\phi)  \\&+d\phi\otimes \nabla_c\nabla_d(d\phi))_{ab}\\ =& g^{ac}g^{bd} \nabla_c\nabla_d(d\phi)_a\phi_b +g^{ac}g^{bd}\nabla_d(d\phi)_a\nabla_c(d\phi)_b\\& + g^{ac}g^{bd} \nabla_c(d\phi)_a \nabla_d(d\phi)_b g^{ac}g^{bd} \phi_a\nabla_c\nabla_d(d\phi)_b \\ =&g^{ac}g^{bd}\nabla_d\nabla_c(d\phi)_a\phi_b + g^{ac}g^{bd}R_{cdae}(d\phi)_mg^{me}\phi_b + |\nabla_g^2\phi|^2 \\ &+(\Delta_g\phi)^2+g^{ac}\phi_a\nabla_c(\Delta_g\phi) \\ =& 2\langle\nabla_g\Delta_g \phi, \nabla_g\phi\rangle + (\text{Ric}_g)_{de}g^{bd}g^{me}\phi_m\phi_b + |\nabla_g^2\phi|^2+ (\Delta_g\phi)^2 \\ =& 2\langle \nabla_g \Delta_g\phi, \nabla_g \phi\rangle + \text{Ric}_g(\nabla_g\phi, \nabla_g\phi) + |\nabla_g^2\phi|^2+ (\Delta_g\phi)^2.
\end{split}
\]
This proves the first equality. The second equality follows from Bochner's formula.


\end{document}